\documentclass[journal,twoside,web]{ieeecolor}
\usepackage{defs}

\begin{document}
\pdfminorversion=9 
\pdfcompresslevel=9
\pdfobjcompresslevel=50

%%% title %%%
\title{Sparse robust optimal control in continuous-time: a computationally viable approach}

\author{Siddhartha Ganguly, Ashwin Aravind, Souvik Das, Masaaki Nagahara, and Debasish Chatterjee
\thanks{The current manuscript was submitted on $8^{\text{th}}$ July, 2025 for review. This work was not supported by any funding agency.}
\thanks{Siddhartha Ganguly is with the Daniel Guggenheim School of Aerospace Engineering, Georgia Institute of Technology, Atlanta, USA (e-mail: \textsf{sganguly41@gatech.edu}).} 
\thanks{Ashwin Aravind is with the Fujitsu Research, Kawasaki, Japan (e-mail: \textsf{aravind.ashwin@fujitsu.com}).}
\thanks{Souvik Das is with the Department of Information Physics and Computing, The University of Tokyo, Japan (e-mail: \textsf{souvikd@g.ecc.u-tokyo.ac.jp}).}
\thanks{Masaaki Nagahara is with the Graduate School of Advanced Science and Engineering,
Hiroshima University, Japan (e-mail: \textsf{nagam@hiroshima-u.ac.jp}).}
\thanks{Debasish Chatterjee is with the Centre for Systems and Control, Indian Institute of Technology Bombay, India (e-mail: \textsf{dchatter@iitb.ac.in}). Debasish Chatterjee acknowledges support of the ANRF grant ANRF/ARGM/2025/001464/MTR from the Government of India.}
}

\maketitle

\begin{abstract}
This article presents a novel, numerically viable algorithm for solving sparse robust optimal control problems in continuous time. We consider a constrained linear noisy system governed by an ordinary differential equation (ODE), with an \(\lpL[1]\)-type objective function in line with the sparse optimal control literature. The resulting optimal control problem is shown to admit a semi-infinite programming (SIP) formulation. Building upon this insight, we develop a new framework that enables the computation of \emph{exact solutions} --- to our knowledge, the first such achievement in the context of sparse optimal control. We demonstrate that a finite and computationally viable convex optimization problem can be solved to recover, in a \emph{lossless manner}, both the optimal value and the corresponding optimizers of the original SIP, while also guaranteeing satisfaction of uncountably many constraints. We also show that the parameter-dependent noisy systems and the minimum attention problem fall into our framework and can be solved efficiently via our algorithm. The efficacy of our algorithm is illustrated through a benchmark numerical example.
\end{abstract}


\section{Introduction}
\label{sec:introduction}
Sparsity and the associated study of numerical techniques to promote sparsity have gained significant prominence across various scientific and technological fields interacting with signal processing \cite{ref:SparseBook:Unser,ref:SparseBook:Vidyasagar}, machine learning \cite{SparseBook:ML:StatLearning}, and statistics \cite{ref:SparseApp:Donoho}. These tools, often referred to as compressed sensing, compressive sampling, sparse representation, or sparse modeling, provide powerful tools for efficient handling of high-dimensional data. 

Sparsity in control theory has gained significant attention due to its ability to reduce control activity and enhance the ability for multitasking. From an application standpoint, sparsity-driven design has been widely adopted in networked control \cite{ref:SparseApp:NetworkPackets}, sensing \cite{ref:SparseApp:SensingFeedDisgn}, aerospace \cite{ref:SparseApp:Space}, and control of partial differential equations \cite{ref:SparseApp:PDEs:I, ref:SparseApp:PDE:II}. In constrained control systems, inducing sparsity by minimizing certain objective functions helps reduce control activity/attention by keeping actuators inactive for extended periods through constant-valued control inputs --- a principle known as \emph{sparse optimal control}. In the framework of deterministic linear systems, the minimization of the \(\lpL[0]\)-(semi)norm of the control input \(\cont(\cdot)\), which effectively reduces the duration over which \(\cont(\cdot)\) stays nonzero, is referred to as \emph{the maximum hands-off control} \cite{ref:MN:MaxHandsOff, ref:DC:MaxHandsOff:SCL}. This approach prioritizes sparsity by inducing prolonged intervals of zero control input, thereby optimizing control activity/attention while ensuring consistent constraint satisfaction. In the sparse control literature, the \(\lpL[0]\) objective is often relaxed to its convex surrogate, the \(\lpL[1]\) norm, under mild assumptions on the problem data \cite{ref:MN:MaxHandsOff, ref:DC:MaxHandsOff:SCL, ref:SidKen-26, ref:ItoIkeKashi-21,ref:ExaEvaPan-18}.

However, sparse control is typically not robust to uncertainties, and ensuring robustness/resilience to alterations in the system models (e.g., parameter mismatch, drift, and/or process noise) is an important task. The motivation for robustness in sparse control is clear: real-world dynamical models contain uncertainties including unmodeled dynamics, process noise, and disturbances etc., all of which can degrade the given system's performance if the uncertainties are not included in the control design at the synthesis stage. While several numerical methods exist for solving sparse optimal control problems in deterministic settings, techniques for inducing robustness along with sparsity have remained an open challenge especially from a numerical standpoint \cite{ref:PolSch-05,ref:PolakMayneSIPsInControl}. This is due to the need to handle a compact yet uncountable family of inequality constraints in a computationally viable way --- an issue that, to date, has not been adequately addressed in the literature and remains an open problem.
\subsubsection*{Our contributions}

Our efforts here are firmly directed at consistent state-action constraint satisfaction despite the uncertainties/noise while a sparsity-inducing \(\lpL[1]\) objective function of the control trajectories is minimized.\footnote{In this article, sparsity is \emph{directly} promoted by the \(\lpL[1]\) cost --- rather than introduced as a surrogate for an \(\lpL[0]\) objective --- and this \(\lpL[1]\)-induced sparsity is a standard and widely used modeling choice in the signal processing and control community~\cite{ref:MV:20}.}

\begin{enumerate}[leftmargin=*]
  
    \item \label{contrib:II} We introduce and structurally formulate a new class of robust continuous-time sparse optimal control problems for linear systems with process noise, parametric uncertainty, hard state, input, and terminal constraints together with an integral \(\lpL[1]\) cost. After parametrizing control and disturbance with a finite dictionary, we show that this class admits a convex semi-infinite programming (CSIP) representation.

   \item  More precisely, the ensuing finitely parametrized OCP continues to stipulate the satisfaction of an uncountable family of constraints. Standard direct methods for optimal control \cite{ref:betts-book, ref:SG:NR:DC:RB-22, ref:gansilcha24} rely on the discretization of the time horizon and enforcement of constraints only at the discretization points. As a result, they offer no guarantees that constraints are satisfied over the entire time interval --- let alone for all possible disturbance realizations. 
   We also draw attention to \cite{ref:elango2024successive, ref:fazlyab2016interior}, where the authors employ successive convexification and interior-point schemes to address continuous-time constraint satisfaction in the noise-free setting. In contrast, for the general noisy case, our approach builds on recent techniques from robust optimization developed in \cite{ref:DasAraCheCha-22, ref:ParCha-23} and the preliminary developments reported in \cite{ref:SD:SG:AM:DC:CtmpcCDC}. Along with these, we leverage the specific convexity properties of the cost and constraint sets to enforce the infinite family of constraints in a numerically viable manner, without introducing conservatism.
   More precisely, we establish guarantees of \emph{exact solutions} for the finitely parametrized constrained optimal control problem --- the optimal value and optimizers obtained via our technique match those of the finitely parametrized yet infinitely constrained problem. We validate our claims through two numerical examples, including one from the domain of minimum attention control \cite{ref:brockett1997minimum, MN-DN:20}. 

\end{enumerate}

\vspace{1mm}

\subsubsection*{Perspectives and questions}

In a certain sense, finite parametrization of the control trajectories (employed herein) reduces the overwhelming complexity associated with infinite-dimensional spaces to finite-dimensional ones, and is inevitable for general constrained problems \cite[Chapter 16, Section 1]{ref:AliTsa-21}, \cite[\S 1]{ref:minmax:ness:vinter}. In our context of linear dynamical systems, the control and noise trajectory parameters enter affinely in the dynamics. It is then natural to ask: are existing sparse robust signal processing techniques with supporting tools from robust optimization, e.g., \cite{ref:BenElGNem-09}, capable of addressing our constrained and uncertain optimal control problem? No, they are not, and primarily for two reasons: (a) The existing tools of sparse robust signal processing employ robust convex optimization problems admitting equivalent standard convex reformulations (typically via duality theory); such classes of problems require specific (typically affine) nature of the uncertainties that are both unnatural and inadequate in \emph{control} problems in contrast to signal processing applications. (b) Often conservative approximations are adopted in the process of employing the indicated type of robust convex optimization problems; in contrast, our techniques lead to exact solutions, and are, by definition, devoid of conservatism --- an important and desirable feature in robust control. A detailed discussion of this matter appears in \S\ref{sec:discussion}.

Finally, in the spirit of the three aspects identified in \cite[see \S VI, Conclusion, third paragraph]{ref:PolakMayneSIPsInControl} --- namely, the introduction of novel results from semi-infinite programming into control, the development of systematic solution techniques, and their practical implementation --- we instantiate these for a robust, sparsity-promoting continuous-time OCP with process and parametric uncertainties. Our main new contributions lie in developing a solution approach for the previously mentioned problem via a convex semi-infinite programming formulation and an exact finite-constraint surrogate, and in providing the associated \(\mathsf{SparseRob}\) solution architecture tailored to this problem class. In particular, we also follow the recommendation outlined in \cite[Page 3, second and third paragraph]{ref:minmax:ness:vinter} regarding the use of parametrization and semi-infinite optimization as a numerical approach for solving robust optimal control problems, building on similar motivations and methodological choices as those outlined in \cite{ref:minmax:ness:vinter}.

We employ standard notations; \(\N = \{1, 2, \ldots\}\) denotes the set of positive integers and \(\Nz = \N \cup \{0\}\). If \(A\) is a subset of \(\Rbb\), then the indicator function \(\indic{A}(\cdot)\) of \(A\) is defined by \(\indic{A}(x) = 1\) if \(x\in A\) and \(\indic{A}(x) = 0\) otherwise. Let us define the family of piecewise constant functions. Consider a fixed horizon \(\horizon > 0\) and a uniform \(K\)-partition: \(0=t_0 <\cdots<t_{K}\) and \( T-t_{K} \le h,\) with step size \(h =t_k-t_{k-1} > 0\), for \(k=1,\ldots,K\). Note that \(K = \left \lceil \frac{T}{h} \right \rceil\) for our case. The space of \(\Rbb\)-valued piecewise constant functions is defined by
\begin{align*}
\pwcfunc_h(\lcrc{0}{\horizon};\Rbb) \Let \hspace{-1mm} \left\{  y(\cdot)   \;\middle\vert\;  
\begin{array}{@{}l@{}}
y(\cdot) = y_{K}\indic{\lcrc{t_{K}}{\horizon}}(\cdot) \\\hspace{-1mm}+  \sum_{k=1}^{K} y_{k-1} \indic{\lcro{t_{k-1}}{t_k}}(\cdot),\, y_k\in\Rbb 
\end{array}
\right\}.
\end{align*}
\noindent By \(\norm{\cdot}_{\ell_1}\) and \(\norm{\cdot}_{\ell_{\infty}}\) we denote the standard Euclidean \(\ell_1\)- and the box-norms.

%%%%%%%%%%% Problem formulation %%%%%%%%%%%
\section{Problem formulation}\label{sec:prob_form}
%%%%%%%%%%%%
Let \(\dimst,\dimcon \in \N\) and fix a time horizon \(\tfin>0\). Let us consider the time-invariant continuous-time noisy linear system 
\begin{equation} \label{eq:sys}
\dot \st(t) = A\st(t) + B\cont(t) + \dist (t) \,\,\text{for all } t \in \lcrc{\tinit}{\tfin},
\end{equation}
%%%%%
with the following data:
%%%%%%%%%%%%%%%%%%%%%%%%%%
%%%%%%% OCP data %%%%%%%%%
%%%%%%%%%%%%%%%%%%%%%%%%%%
\begin{enumerate}[label=\textup{(\alph*)}, leftmargin=*, widest=b, align=left]
\item \label{ocp:data:1} the state trajectory \(t\mapsto x(t) \in \Rbb^d\) is absolutely continuous, the control and the disturbance trajectories are locally integrable maps \(t \mapsto \cont(t) \in \Rbb^{\dimcon}\) and \(t \mapsto \dist(t) \in \Rbb^{\dimw}\), respectively; \(A \in \Rbb^{\dimst \times \dimst}\) and \(B \in \Rbb^{\dimst \times \dimcon}\).
\item \label{ocp:data:2} For all \(t \in \lcrc{0}{\horizon}\), instantaneous constraints are imposed on the control actions, which are of the form
\begin{equation}
    \label{eq:inst_control}
    \cont(t) \in \admcont \Let \prod_{i=1}^{\dimcon} \admcont_i = \prod_{i=1}^{\dimcon} \lcrc{-\ubound^i}{\ubound^i} \subset \Rbb^{\dimu},
\end{equation}
where \(\ubound^i>0\) for each \(i=1,2,\ldots,\dimu\). Moreover, for all \(t \in \lcrc{0}{\horizon}\), the admissible disturbances are assumed to satisfy
\begin{equation}
    \label{eq:inst_distur}
    \dist(t) \in \admdist \Let \prod_{j=1}^{\dimx} \admdist_j = \prod_{j=1}^{\dimx} \lcrc{-\wbound^j}{\wbound^j} \subset \Rbb^{\dimx},
\end{equation}
with \(\wbound^j>0\) for each \(j=1,2,\ldots,\dimx\). For all \(t \in \lcrc{\tinit}{\tfin}\), the states will be required to take values in \(\admst \subset \Rbb^{\dimst}\), a compact and convex set containing \(0\in \Rbb^d\) in its interior. We remark that our subsequent analysis also extends to general compact and convex admissible sets \(\admcont \subset \Rbb^{\dimcon}\) and \(\admdist \subset \Rbb^{\dimx}\); the hyperrectangle structure in \eqref{eq:inst_control} and \eqref{eq:inst_distur} was assumed purely for simplicity.

\item \label{ocp:data:3} The initial state \(x(0) \Let \param \in \Rbb^d\) is given and the final state \(\st(\horizon)\) takes values in the terminal set \(\finset\), which is a compact and convex subset of \(\admst\) containing the origin \(0\in \Rbb^d\) in its interior. 
\end{enumerate}
\vspace{1mm}
%%%%%%%%%

%%%%%%%%%%%%%%%%%%%%
%%%%%%%%%%%%%%%%%%%%
%%%%%%%%%%%%%%%%%%%%
\textbf{The \(\lpL[1]\)-sparse robust optimal control:}
A widely used approach to introduce sparsity in control design is to employ the standard sparsity-promoting \(\lpL[1]\)-norm \cite{ref:MH-book-20}. Given a time horizon \(\horizon>0\), we consider the objective function
\begin{equation}\label{eq:cost_l1}
  \cont(\cdot) \mapsto  \objective(\cont(\cdot)) \Let  \int_{\tinit}^{\tfin} \norm{\cont(t)}_{\ell_1}\odif{t},
\end{equation}
and over the feasible control trajectories \(\cont(\cdot)\) we consider the sparse robust continuous-time OCP
\begin{equation}
	\label{eq:SR-OCP} 
\begin{aligned}
& \inf_{\cont(\cdot)} 	&&  \objective(\cont(\cdot))  \\
&  \sbjto		&&  \begin{cases}
\text{dynamics}\,\,\eqref{eq:sys},\,\st(\tinit)= \param,\, \st(\tfin) \in \finset, \text{ and}\\ 
\st(t) \in \admst, \, \cont(t) \in \admcont \; \text{for all }t \in 
\lcrc{0}{\horizon} \text{ and} \\
\text{all }w(t) \in \admdist.
\end{cases}
\end{aligned}
\end{equation}
The OCP \eqref{eq:SR-OCP} is an infinite-dimensional optimization problem and is numerically intractable in general. Indeed, the minimization is performed over an infinite-dimensional set of feasible controls, and all the constraints in \eqref{eq:SR-OCP} need to hold for all \((t,\dist(\cdot))\). We assume that the OCP \eqref{eq:SR-OCP} is feasible and admits a solution. To ensure numerical viability, we begin by reducing the infinite-dimensional problem to a finite-dimensional one by parametrizing the admissible sets of controls and disturbances --- a standard approach in numerical optimal control \cite{ref:minmax:ness:vinter}, but we shall retain the uncountable family of constraints in \eqref{eq:SR-OCP}.

\begin{rem}\label{rem:inf-sup is inf}
Note that since the objective \(\objective(\cont(\cdot))\) in \eqref{eq:SR-OCP} does not have a state-dependent term (which depends on the disturbance and uncertainties), the disturbance will not influence the value of the objective. Consequently, introducing a supremum over the disturbance trajectories in the cost would not change the optimization problem or the optimal value. Therefore, \(\inf_{u(\cdot)} \sup_{w(.)} \mathsf{J}(u(\cdot)) = \inf_{u(\cdot)} \mathsf{J}(u(\cdot))\)
holds.
\end{rem}

%%%%%%%%%%%%%%%%%%%%%%%%%%%%%%
%%%%% Tract. reform %%%%%%%%%%
%%%%%%%%%%%%%%%%%%%%%%%%%%%%%%
\subsection{Viable reformulation: control and disturbance parametrization}\label{subsec:pwc:disc}
We adopt the following parametrizations:
\begin{defn}\label{defn:discrete_admcon}
\longthmtitle{Parameterization for control trajectories}
Let \(N \in \N\) and consider a finite dictionary \(\dict \Let \aset[]{\psi_i(\cdot)}_{i=1}^{N} \subset  \pwcfunc_h(\lcrc{0}{\horizon};\Rbb)\) consisting of piecewise constant functions that are linearly independent and \(\max_{i} \max_{t}\abs{\dicC_i(t)} =1\). We define the parametrized set of (\(\text{i}\)-\({\text{th}}\) component) admissible control trajectories \(\admconfunc_{\dict}  \subset \pwcfunc_h(\lcrc{0}{\horizon};\Rbb)\) by 
\begin{align*}
    \admconfunc_{\dict} \Let \linspan \aset[\big]{ \psi_i:\lcrc{0}{\horizon} \lra \Rbb \suchthat i = 1, \ldots, N }.
\end{align*}  
Let \(t \mapsto \Reg(t) \Let \bigl(\reg_1(t)\; \reg_2(t)\;\ldots \;\reg_N(t) \bigr) \in \Rbb^{N}\). Then, the control signal is written as
\begin{align}\label{eq:disct:control}
\lcrc{\tinit}{\tfin} \ni t \mapsto \cont^{\dict}_i(t) = \sum_{j=1}^{N} \Param_{i, j} \psi_j(t) = \inprod{\Param_i}{\Reg(t)}
\end{align}
for each \(i \in \aset[]{1,\ldots,\dimcon}\), where \(\Param_{i} \Let (\Param_{i,1},\ldots,\Param_{i,N}) \in \Rbb^{N}\) are the control parameters. 
\end{defn}
Defining \( \Param \Let \bigl(\Param_{i,j}\bigr)_{i=1,j=1}^{\dimcon,N}\in\Rbb^{\dimcon \times N}\), the parametrized control trajectory in \eqref{eq:disct:control} can be written compactly as \(
t \mapsto \cont^{\dict}(t) = \Param \Reg(t)
\).
\begin{defn}
\label{defn:discrete_admdist}
\longthmtitle{Parameterization for disturbance trajectories}
Let \(M\in \N\) and consider a finite dictionary \(\mathcal{D}_w \Let \aset[]{\dicD_j (\cdot)}_{j =1}^M \subset \pwcfunc_h(\lcrc{0}{\horizon};\Rbb)\) consisting of piecewise constant functions that are linearly independent and \(\max_{j} \max_{t} \abs{\dicD_j(t)} =1\). We define the parametrized set of (\(\text{j}\)-\({\text{th}}\) component) admissible disturbances \(\admdistfunc_{\dict} \subset \pwcfunc_h(\lcrc{0}{\horizon};\Rbb)\) by 
\begin{align*}
    \admdistfunc_{\dict} \Let \linspan \aset[\big]{ \dicD_j:\lcrc{0}{\horizon} \lra \Rbb \suchthat  j = 1, \ldots, M},
\end{align*}
yielding the following \emph{finite} parametrization for each component of the admissible disturbance trajectories
\begin{align}\label{eq:disct:dist}
    \lcrc{\tinit}{\tfin} \ni t \mapsto w^{\dict}_i(t)= \sum_{j=1}^{M}\Paramw_{i,j} \dicD_j(t) = \inprod{\Paramw_i}{\RegD(t)}, 
\end{align}
where \(i=1,\ldots,\dimx\), \(\Paramw_i \Let \bigl(\Paramw_{i,1},\ldots, \Paramw_{i,M}\bigr) \in \Rbb^{M}\), and \(t \mapsto \RegD(t) \Let \bigl(\dicD_1(t),\ldots,\dicD_{M}(t)\bigr) \in \Rbb^{M}\).
\end{defn} 
We represent the admissible disturbance trajectories compactly as \(\lcrc{0}{\horizon} \ni t \mapsto \dist^{\dict} (t) = \Paramw  \RegD(t)\) with \(\Paramw \Let (\Paramw_1 \, \Paramw_2 \, \cdots \, \Paramw_{\dimx}) \in \Rbb^{\dimx \times M}\) and \(\Paramw = \bigl(\Paramw_{i,j} \bigr)_{i=1,j=1}^{d,M}\). 
%%%%% remark on PWC parametrization %%%%%%%
\begin{rem}\label{rem:pwc}
\longthmtitle{Choice of generating functions}
The selection of the generating functions in Definitions \ref{defn:discrete_admcon} and \ref{defn:discrete_admdist}, and the choice of \(N\) should be guided by the specific goals and sampling frequency of the actuator hardware, especially when the time horizon is prescribed. We further emphasize that \(N\) is not entirely a free numerical parameter; rather, it is tied to the temporal resolution at which the control input can be implemented in the application at hand.
In particular, a piecewise constant dictionary is often ideal for digital implementations and sparse-control scenarios, since its simple and quantized structure aligns naturally with fixed‐precision hardware and enables efficient representation when only a few basis elements are active at any given time. Moreover, since admissible control (and disturbance) trajectories are locally integrable and the family of piecewise constant functions is dense in \(\lpL[1](\lcrc{0}{\horizon};\admcont)\) (and \(\lpL[1](\lcrc{0}{\horizon};\admdist)\)) \cite[Chapter~6]{ref:FolReal-99}, it follows that, for sufficiently large \(N\) (and \(M\)), any such trajectory admits an arbitrarily close approximation (in the \(\lpL[1]\)-norm) by elements of \(\mathsf{PC}_h(\lcrc{0}{\horizon};\Rbb)^{\dimcon}\) (and \(\mathsf{PC}_h(\lcrc{0}{\horizon};\Rbb)^{\dimst}\)). Similar arguments can also be made for the disturbance trajectories.
\end{rem}

%%%%%%%%%%%%%%%%%%%%%%%%%%%%%%%%%%%%%%
%%%%%%% Prop: feasible sets %%%%%%%%%%
%%%%%%%%%%%%%%%%%%%%%%%%%%%%%%%%%%%%%%
We now analyze structural properties of the admissible sets for control and disturbance parameters in the following proposition.
\begin{prop}\label{prop:comp:conv:adparam:sets}
\longthmtitle{Properties of admissible sets for control and disturbance parameters}
Consider the OCP \eqref{eq:SR-OCP} and the parametrizations \eqref{eq:disct:control} and \eqref{eq:disct:dist}. Then
\begin{enumerate}[label=\textup{(S-\alph*)}, leftmargin=*]
\item \label{eq:adm:con:set} the set of admissible control parameters corresponding to the representation \eqref{eq:disct:control}
%%%%% control feasible set %%%%%%
\begin{equation}\label{eq:ad_cont_traj}
\adparam \Let   \left\{\Param \in \Rbb^{\dimcon \times N} \, \middle\vert 
\begin{array}{@{}l@{}}
     \, \Param \Reg(t) \in \admcont \; \text{for all }t \in \lcrc{\tinit}{\tfin}
        \end{array}
        \right\},\nn
\end{equation}
\item \label{adm:dist:set} and the set of admissible disturbance parameters corresponding to the representation \eqref{eq:disct:dist} 
%%%%% disturb feasible set %%%%%%
\begin{align}
    \adparamD \Let \aset[\Big]{\Paramw \in \Rbb^{d\times M} \suchthat \Paramw \RegD(t)  \in \admdist \; \text{for all }t \in \lcrc{\tinit}{\tfin}}\nn
\end{align}
\end{enumerate}
are compact and convex. 
\end{prop}
% A proof of Proposition \ref{prop:comp:conv:adparam:sets} is in Appendix~\ref{appendix:A}.
%%%%%%%%%%%%
\begin{proof}
Recall from Definition \ref{defn:discrete_admcon} and \ref{eq:disct:control} that a finite parametrization of the control actions leads to the expression
\(\lcrc{0}{\horizon}\ni t\mapsto\cont_i^{\dict}(t) = \inprod{\Param_i}{\Reg(t)} \in \admcont_i\) for \(i=1,\ldots,m\). Since \(\admcont_i \Let \lcrc{-\ubound^i}{\ubound^i}\), we write \(\inprod{\Param_i}{\Reg(t)}  \in \lcrc{-\ubound^i}{\ubound^i}\). This is equivalent to the inequality
\begin{equation}\label{eq:contraints}
    \abs{\inprod{\Param_i}{\Reg(t)}} \leq \ubound^i \text{ for all }t \in \lcrc{0}{\horizon},
\end{equation}
and H\"older's inequality \cite[p.\ 128]{ref:VAZZ-16} gives us \(\abs{\inprod{\Param_i}{\Reg(t)}} \leq \norm{\Param_i}_{\ell_1} \norm{\Reg(t)}_{\ell_\infty}\) for all \(t\in\lcrc{0}{\horizon}\). Therefore, since \(\max_{i} \max_{t}\abs{\dicC_i(t)} =1\) by Definition \ref{defn:discrete_admcon}, stipulating that \(\norm{\Param_i}_{\ell_1}  \leq \ubound^i \text{ for }t \in \lcrc{0}{\horizon}\) and \(i=1,\ldots,m,\) ensures that \eqref{eq:contraints} is satisfied for all \(t \in \lcrc{0}{\horizon}\). Observe that:
\begin{enumerate}[label=\textup{P.\alph*)}, leftmargin=*]
    \item The mapping \(\Param \mapsto \norm{\Param_i}_{\ell_1} - \ubound^i = \sum_{j=1}^{N}\abs{\Param_{i,j}} - \ubound^i\) is both convex and continuous. Now, the convexity of \(\adparam\) is immediate because the preimage of a convex set under a linear map is convex.
    
    \item Since \(\admcont\) is compact (and therefore closed), the set \(\adparam\) is also closed. This is because \(\adparam\) is the preimage of a closed set under a continuous map.

    \item Exploiting the structure of \(\admcont\), for \(i=1,\ldots,\dimcon\), we have \(\Param_i \Reg(t) \in \lcrc{-\ubound^i}{\ubound^i}\) for all \(t \in \lcrc{0}{\horizon}\). For \(j \in \aset[]{1,\ldots,N}\), choose \(t \in \lcro{t_j}{t_{j+1}}\), then due to the piecewise constant nature of the functions \(\psi_i(\cdot)\) we have \(\abs{\Param_{i,j}} \le \ubound^i\). The boundedness of \(\adparam\) follows immediately. 

\end{enumerate}
%%%%%
\vspace{1mm}
By the Heine–Borel theorem \cite[Theorem 2.41, p.\ 40]{ref:Rud-Analysis}, \(\adparam\) is compact. Convexity and compactness of the set \(\adparamD\) follow similarly. Our proof is complete.
\end{proof}

Using the parametrizations in \eqref{eq:disct:control} and \eqref{eq:disct:dist}, we express the solution to \eqref{eq:sys} as
%%%%
\begin{align}\label{eq:p_sol}
\lcrc{0}{\horizon} \ni t \mapsto \st(t; \param, \Param, \Paramw) &\Let \epower{At} \overline{x} + \int_{0}^{t} \epower{A(t-\tau)} B \Param \Reg(\tau) \, \mathrm{d}\tau \notag \\ 
&\quad + \int_{0}^{t} \epower{A(t-\tau)} \Paramw \RegD(\tau) \, \mathrm{d}\tau.
\end{align}
%%%%%%%%%%%
We occasionally denote the map in \eqref{eq:p_sol} by \(t \mapsto \st(t)\), omitting the parameters \(\param\), \(\Param\), and \(\Paramw\) for notational simplicity. Moreover, observe that the
generating functions \(\psi_j(\cdot)\), for \(j=1,\ldots,N\), are piecewise constant (see Definition \ref{defn:discrete_admcon}) and consequently we have 
\begin{align}    
  \Rbb^{N} \ni  \Reg(t) = \begin{cases}
        \Reg_{k-1} \quad &\text{if }t\in \lcro{t_{k-1}}{t_k},\\
        \Reg_K & \text{if }t\in \lcrc{t_K}{T},
    \end{cases}
\end{align}
where \(\Reg_{k-1}\) for \(k=1,\ldots,K-1\) and \(\Reg_{K}\) are all vectors in \(\Rbb^{N}\). The cost in \eqref{eq:SR-OCP} can therefore, be further simplified as
\begin{align}\label{eq:simplified_cost}
    \objective\bigl(\Param\Reg(\cdot)\bigr) &\Let \int_{0}^{\horizon} \norm{\Param \Reg(t)}_{\ell_1} \odif{t} \nn\\
    &= \sum_{k=1}^{K} \int_{t_{k-1}}^{t_k} \norm{\Param \Reg_{k-1}}_{\ell_1}\odif{t} + \int_{t_K}^{\horizon}\norm{\Param\Reg_{K}}_{\ell_1}\odif{t}  \nn \\
    &  =\sum_{k=1}^{K}h\norm{\Param\Reg_{k-1}}_{\ell_1}+ (\horizon-t_K) \norm{\Param\Reg_K}_{\ell_1},
\end{align}
where \(h=t_k-t_{k-1}>0\).
%%%%%%%%%%%%%%
Putting everything together in \eqref{eq:SR-OCP}, we obtain the following optimization problem:
 \begin{equation}
\label{eq:SR:pre_SIP} 
\begin{aligned}
& \hspace{-4mm}\inf_{\Param \in \adparam} 	&& \sum_{k=1}^{K}h\norm{\Param\Reg_{k-1}}_{\ell_1}+ (\horizon-t_K) \norm{\Param\Reg_K}_{\ell_1} \\
&  \hspace{-4mm}\sbjto		&&  \begin{cases}
\st(0)=\xz,\,\st(\tfin) \in \finset,\,\st(t) \in \admst, \text{ and}\\
\Param\Reg(t)\in \admcont \text{ for all }(t,\Paramw) \in \lcrc{\tinit}{\tfin} \times \adparamD.
\end{cases}
\end{aligned}
\end{equation}   
Although \eqref{eq:SR:pre_SIP} is a finite-dimensional parametrized avatar of \eqref{eq:SR-OCP}, it is still a semi-infinite program consisting of an uncountable family of constraints indexed by \((t,\Paramw) \in \lcrc{\tinit}{\tfin} \times \adparamD\). For simplicity of notation, we continue to use \(\objective(\Param\Reg(\cdot))\) to denote the objective function, instead of the full expression in \eqref{eq:simplified_cost}.

\section{Main results}
\label{sec:main_result}
%%%%%%%%
This section focuses on solving  \eqref{eq:SR:pre_SIP}, which turns out to be a convex semi-infinite program (CSIP) in the decision variable \(\Param \in \adparam\). Let us first revisit key points outlined in the introduction before advancing to our theoretical results. Despite the finite parametrization in \S\ref{sec:prob_form}, the OCP \eqref{eq:SR:pre_SIP} continues to stipulate the satisfaction of an uncountable family of state and control constraints. From a numerical standpoint, this aspect remains highly challenging, and to the best of our knowledge, no existing method has provably/effectively addressed this issue. Nonetheless, our results operate entirely within a finitary computational framework, yet guarantee exact solutions and satisfaction of an uncountable family of constraints robustly. To achieve this, we construct a family of regularized problems for the OCP \eqref{eq:SR:pre_SIP} and develop an algorithmic framework that leverages recent advances in the numerical solution of CSIPs from \cite{ref:DasAraCheCha-22,ref:ParCha-23}, enabling a computationally viable approach.

\smallskip

Let \(\regularizer:\Rbb^{\dimcon \times N} \lra \Rbb\) be a continuous, positive, and strictly convex function, and let \(\eps>0\) be a parameter. Consider the \((\eps,\regularizer)\)-regularized version of the OCP \eqref{eq:SR:pre_SIP}
\begin{align}\label{eq:regularized:SIP}
 \inf_{\Param \in \adparam} \hspace{0mm}\left\{\objective\bigl(\Param\Reg(\cdot)\bigr) + \eps \regularizer(\Param) \middle\vert  
\begin{array}{@{}l@{}}
\text{constraints in }\eqref{eq:SR:pre_SIP} \text{ holds}  \\ \forall(t,\Paramw) \in \lcrc{\tinit}{\tfin} \times \adparamD
\end{array}
\right\}.
\end{align}
For a fixed \(\xz\) and \(\regupara>0\), define the optimal value of \eqref{eq:regularized:SIP} by \(\valuefunc(\xz,\regupara)\). Fix \(\eps > 0\); we denote the set of feasible initial states \(\fsblset\) for \eqref{eq:regularized:SIP} by
\begin{equation}\label{eq:feas_set}
    \fsblset \Let \aset[\big]{\dummyx \in \Rbb^{\dimst} \suchthat \valuefunc({\dummyx},\eps) < + \infty}.
\end{equation}
First, we show that \eqref{eq:regularized:SIP} and its feasible set possess some good qualitative properties. 
\begin{prop}\label{prop:unique solution}
\longthmtitle{Existence and uniqueness of solutions for~\eqref{eq:SR:pre_SIP} and~\eqref{eq:regularized:SIP}}
Assume that the feasible set of \eqref{eq:SR:pre_SIP} defined by
\begin{align}
\mathcal{F}_{O} \Let \left\{ \Param \in \adparam \;\middle\vert\;  
\begin{array}{@{}l@{}}
\st(0)=\xz,\,\st(\tfin) \in \finset,\,\st(t) \in \admst\\
\text{ for all }(t,\Paramw) \in \lcrc{\tinit}{\tfin} \times \adparamD
\end{array}
\right\}\nn
\end{align}
is nonempty. Then, the program \eqref{eq:SR:pre_SIP} admits a solution, and its regularized counterpart \eqref{eq:regularized:SIP} admits a unique solution for every \(\eps>0\).
\end{prop}

%%%%%%%%% Proof of Proposition III.1 %%%%%%%%%%%%%%
\begin{proof}
Recall the expression 
\begin{equation}
\adparam \Let   \left\{\Param \in \Rbb^{\dimcon \times N} \, \middle\vert 
\begin{array}{@{}l@{}}
     \, \Param \Reg(t) \in \admcont \; \text{for all }t \in \lcrc{\tinit}{\tfin}
        \end{array}
        \right\},\nn
\end{equation}
and consider the feasible set of \eqref{eq:SR:pre_SIP}
\begin{align}\label{eq:feas:set}
\mathcal{F}_{O} \Let \left\{ \Param \in \adparam \;\middle\vert\;  
\begin{array}{@{}l@{}}
\st(0)=\xz,\,\st(\tfin) \in \finset,\,\st(t) \in \admst\\
\text{ for all }(t,\Paramw) \in \lcrc{\tinit}{\tfin} \times \adparamD
\end{array}
\right\}. 
\end{align}
We have the following steps:
\begin{enumerate}[label=\textup{(\alph*)}, leftmargin=*]
\item \label{prop:uq:sol:I} Recall that \(\fsblset\) denotes the set of feasible initial states, as defined in \eqref{eq:feas_set}. Fix \(\xz \in \fsblset\) and \((t,\Paramw) \in \lcrc{0}{\horizon} \times \adparamD\), and define
\begin{align*}
\hspace{-8mm} \auxset^{t,\Paramw}(\param)  \Let \left\{ \Param \in \adparam \;\middle\vert\;  
\begin{array}{@{}l@{}}
       \st(0)=\xz,\,\st(\tfin) \in \finset,\,\st(t) \in \admst
\end{array}
\right\}; 
\end{align*}
notice that \(\Paramw\) enters the constraint \(\st(t) \in \admst\) through \eqref{eq:p_sol}. The sets \(\admst\) and \(\finset\) are compact and convex, and the mapping \(\adparam\ni \Param \mapsto \st(s;\param,\Param,\Paramw)\) is continuous and affine. Consequently, 
\begin{align}
    \auxset^{t,\Paramw}(\param) = \st(s;\param,\cdot,\Paramw)\inverse\bigl(\aset[]{\admst}\bigr) \cap \st(\horizon;\param,\cdot,\Paramw)\inverse\bigl(\aset[]{\finset}\bigr), \nn
\end{align}
is closed and convex, being the preimage of closed and convex sets under continuous affine mappings. Then, 
\begin{align}\label{eq:state:feasible:set}
    \auxset(\param) \Let \bigcap_{(t,\Paramw) \in \lcrc{0}{\horizon} \times \adparamD} \auxset^{t,\Paramw}(\param)
\end{align}
is closed and convex. The feasible set \(\mathcal{F}_{O}\) is then closed and convex since it is the intersection \(\mathcal{F}_{O} = \auxset(\param) \cap \adparam\). From Proposition \ref{prop:comp:conv:adparam:sets}, we also know that \(\adparam\) is compact, and hence \(\mathcal{F}_{O}\) is compact, being a closed subset of the compact set \(\adparam\) \cite[Theorem 2.35, p.\ 37]{ref:Rud-Analysis}.

\item \label{prop:uq:sol:II} Recall the expression of the objective function 
\begin{align}
    \objective\bigl(\Param\Reg(\cdot)\bigr) &=\sum_{k=1}^{K}h\norm{\Param\Reg_{k-1}}_{\ell_1}+ (\horizon-t_K) \norm{\Param\Reg_K}_{\ell_1}.\nn
\end{align}
In view of the above, the map \(\Rbb^{m\times N} \ni \Param \mapsto \objective\bigl(\Param\Reg(\cdot)\bigr) \Let  \int_{0}^{\horizon} \norm{\Param\Reg(t)}_{\ell_1} \odif{t} \in \Rbb\) is continuous and convex.
Since \(\Param \mapsto \norm{\Param\Reg_k}_{\ell_1}\) is continuous and convex\footnote{Note that if \(x\mapsto f_1(x) \Let \norm{x}\) and \(\mathsf{M}\) is a finite-dimensional linear operator, then \(x \mapsto f_2(x) \Let f_1(\mathsf{M}\Param)\) is convex, where \(\norm{\cdot}\) is any Euclidean norm.},
the mapping \(\Rbb^{m\times N} \ni \Param \mapsto \int_{0}^{\horizon} \norm{\Param\Reg(t)}_{\ell_1} \odif{t} \in \Rbb\) is a finite sum of continuous and convex functions, implying that it is continuous and convex. 
\end{enumerate}
\vspace{2mm}
Thus, the existence of solution of the CSIP \eqref{eq:SR:pre_SIP} follows immediately from Weierstrass theorem \cite[Chapter \(9\), p. \(33\)]{ref:VAZZ-16}.

For the regularized SIP \eqref{eq:regularized:SIP} fix \(\eps>0\) and notice that the objective 
\begin{align}
    \Param \mapsto \objective \bigl(\Param\Reg(\cdot)\bigr) + \eps \regularizer(\Param) \nn
\end{align}
is continuous and strictly convex. Indeed, continuity and convexity of the map \(\Param \mapsto \objective(\Param\Reg(\cdot))\) follow from the preceding arguments and by the choice that \(\Param \mapsto\regularizer(\Param)\) is continuous and strictly convex. Thus \(\Param \mapsto \objective \bigl(\Param\Reg(\cdot)\bigr) + \eps \regularizer(\Param)\) is strictly convex and the feasible set \(\mathcal{F}_O\) is compact and convex. The existence of a unique minimizer follows immediately. 
\end{proof}

% A proof of Proposition \ref{prop:unique solution} is in Appendix~\ref{appendix:A}.

\begin{rem}\label{rem:on:regularizing}
\longthmtitle{Importance of regularized problem~\eqref{eq:regularized:SIP}}
It is worth noting that if the sole objective is the lossless extraction of the optimal value, the regularization of the OCP \eqref{eq:SR:pre_SIP} --- that is, the construction of the regularized problem \eqref{eq:regularized:SIP} --- is, in fact, unnecessary. The architecture \(\sprob(\cdot,\cdot)\) presented ahead in \S\ref{subsec:sparserob}, with \(\eps=0\), can be directly applied to \eqref{eq:SR:pre_SIP} to recover the exact optimal value. However, we emphasize that in this unregularized setting, the recovery of optimizers --- essential for the synthesis of sparse control actions --- is not guaranteed, owing to the lack of strict convexity in the objective function. To address this, we introduce the regularized CSIP \eqref{eq:regularized:SIP}, thereby ensuring a strictly convex objective function that facilitates the lossless extraction of optimizers as well. A numerical illustration of this process can be seen Fig.~\ref{fig:msap} in \S\ref{sec:NumExp}.
\end{rem}
%%%
\subsection{Technical results}\label{sec:tech:results}
We fix the notation 
\begin{align}\label{eq:dvar}
\dvar \Let \dimcon N = \text{dim. of the decision space of \eqref{eq:SR:pre_SIP}},
\end{align}
and define the sequences \(\tseq \Let (t^1,\ldots,t^{\dvar}) \in \lcrc{\tinit}{\tfin}^{\dvar},\, \bseq \Let \bigl( \Paramw^1,\ldots, \Paramw^{\dvar}\bigr) \in \adparamD^{\dvar}.\) Fix \(\eps>0\) and \(\xz \in \fsblset\), and define the relaxed version of \eqref{eq:regularized:SIP} by \[\lcrc{\tinit}{\tfin}^{\dvar} \times \adparamD^{\dvar} \ni  (\tseq,\bseq) \mapsto \gfunc(\xz,\eps;\tseq,\bseq)\in \Rbb,\] 
where \(\gfunc(\xz,\eps;\cdot,\cdot)\) is given by 
\begin{equation}
    \label{eq:g_func}
    \begin{aligned}
        &\hspace{-3mm}\gfunc(\xz, \eps;\tseq,\bseq) && \Let\\  &\hspace{-3mm}\inf_{\Param\in \adparam}	&& \hspace{-13mm} \objective\bigl(\Param\Reg(\cdot)\bigr) + \eps \regularizer(\Param) \\
        &  \hspace{-2mm}\sbjto &&  \hspace{-13mm}\begin{cases}
        \st(0) = \xz, \st(t^i) \in \admst,\,\st(\tfin) \in \finset,\\
        \Param\Reg(t^i) \in \admcont\text{ where } i=1,\ldots,\dvar \\\text{for }(\tseq,\bseq) \in \lcrc{\tinit}{\tfin}^{\dvar} \times \adparamD^{\dvar}.
        \end{cases}
    \end{aligned}
\end{equation}
Recall from the expression \eqref{eq:p_sol} that \(\st(t^i) = \st(t^i;\param,\Param,\Paramw^i)\) for each \(i \in \aset[]{1, \ldots,\dvar}\).

%%%%%%%%%%%%%%%%%%%%%%%%%%%%%%%%%%%%%%%%%
We introduce a key technical assumption that underpins our main results.
%%%% Slater assumption %%%%
\begin{assum}\label{assum:slater-like}
\longthmtitle{Strict feasibility}
The set of feasible initial states \(\fsblset\) defined in \eqref{eq:feas_set} is nonempty and the OCP \eqref{eq:regularized:SIP}, satisfies a \emph{strict feasibility condition} of the following form: there exist \(\Param\), and an open, nonempty set \(O \subset \intr{\admst} \times \intr{\adparam}\) such that the following hold:
\begin{align*}
\begin{cases}
(\st(t),\Param\Reg(t)) \in O, \quad \st(0)=\param, \quad \st(\tfin) \in \intr{\finset}\\
\text{for all } \Paramw \in \adparamD \text{ and all } t \in \lcrc{0}{\horizon}.
\end{cases}
\end{align*}
\end{assum}
%%%
The following Proposition establishes a regularity property of the function \(\gfunc(\cdot)\) to be used in our subsequent analysis.
%%%% Prop on reg of F(.) %%%%%%%
\begin{prop}
\label{lem:exact:sols}
\longthmtitle{Regularity of the function \(\gfunc(\cdot)\)}
Fix \(\regupara>0\) and \(\xz \in \fsblset\), and consider the OCP \eqref{eq:regularized:SIP} and \eqref{eq:g_func} along with their data. Let \(\valuefunc(\xz,\regupara)\) be the optimal value of the OCP \eqref{eq:regularized:SIP} and suppose that the strict feasibility assumption \ref{assum:slater-like} holds for \eqref{eq:regularized:SIP}. Define \(\totconset^{\dvar} \Let \lcrc{0}{\horizon}^{\dvar}\times \adparamD^{\dvar}\). Consider the optimization problem 
    \begin{equation}
    \label{e:global_max_prob}
    \begin{aligned}
        & \sup_{(\tseq,\bseq)}
        &&\gfunc(\xz,\ol{\eps};\tseq,\bseq)\\
        & \sbjto   && (\tseq,\bseq) \in \totconset^{\dvar}.
    \end{aligned}
\end{equation}
    Then:
    \begin{enumerate}[label=\textup{(\ref{lem:exact:sols}-\alph*)}, leftmargin=*, widest=b, align=left]
        \item \label{lemm:exact:sols_1} for each fixed \(\xz \in \fsblset\) and \(\regupara>0\), \(\totconset^{\dvar} \ni  (\tseq,\bseq) \mapsto \gfunc(\xz,\ol{\eps};\tseq,\bseq) \in \Rbb\) is continuous;
% \item \label{thrm:exact_sol_cont}

        \item \label{lemm:exact:sols_2} there exists  \(\bigl(\tseq^{\ast}(\xz,\ol{\eps}),\bseq^{\ast}(\xz, \ol{\eps})\bigr) \in \totconset^{\dvar}\) that solves \eqref{e:global_max_prob}.
    \end{enumerate}
\end{prop}

%%%%%%%%%% Proof of Proposition III.4 %%%%%%%%%%%%%%%

\begin{proof}
\normalfont
We start with the first assertion. Fix \(\regupara>0\), \(\xz \in \fsblset\), and \((\ol{\Param}, \ol{\tseq},\ol{\bseq}) \in \adparam \times \totconset^{\dvar}\). Observe that
\begin{enumerate}[label=\textup{L.\alph*)},leftmargin=*]
    \item \label{lemm:point:a} the map \(\Param \mapsto \int_0^{\horizon} \norm{\Param \Reg(s)}_{\ell_1}\odif{s} + \regupara \regularizer(\Param) \) is convex. Indeed, \(\Param \mapsto \int_0^{\horizon} \norm{\Param \Reg(s)}_{\ell_1}\odif{s}\) is convex; see the arguments in point \ref{prop:uq:sol:II} in Proposition \ref{prop:unique solution}. The map \(\Param \mapsto \regularizer(\Param)\) is convex by design.

    \item \label{lemm:point:b} \((\Param,\tseq,\bseq) \mapsto \int_0^{\horizon} \norm{\Param \Reg(s)}_{\ell_1}\odif{s} + \regupara \regularizer(\Param) \) is continuous around \((\ol{\Param},\ol{\tseq},\ol{\bseq}) \in \adparam \times \totconset^{\dvar}\). The proof of this claim proceeds along the same lines as the arguments presented in the proof of Proposition  \ref{prop:unique solution}, and also uses the fact that the mapping \((\Param,\tseq,\bseq) \mapsto \regularizer(\Param)\) is continuous.

    \item \label{lemm:point:c} Recall the expression of the solution trajectory \eqref{eq:p_sol}. Leveraging the piecewise constant nature of \(\Reg(\cdot)\) and \(\RegD(\cdot)\), we rewrite the map \(t\mapsto x(t)\) as
    \begin{align}
        &\st(t;\param,\Param,\Paramw) \nn \\
        & = \epower{At}\overline{x} +  \int_{0}^{t} \epower{A(t-s)} \bigl(B \Param \Reg(s) + \Paramw \RegD(s)\bigr) \odif{s}\nn 
        \\ & = \epower{At}\overline{x} + \sum_{i=1}^{K}\biggl(\int_{s_{i-1}}^{s_i}\epower{A(t-s)}\odif{s}\biggr)(B\Param\Reg_i+\Paramw\RegD_i).\nn
\end{align}
Thus,
    the constraint map \[ \adparam \times \lcrc{0}{\horizon} \times \adparamD  \ni (\Param,t,\Paramw) \mapsto \begin{pmatrix}
        \st(t; \xz,\Param, \Paramw)\\
        \st(\horizon; \xz,\Param, \Paramw)
    \end{pmatrix}\] 
    is jointly continuous. Moreover, from Proposition \ref{prop:comp:conv:adparam:sets}, the constraints on the control action
    \[\Param \Reg(t) \in \admcont\]
    can be equivalently represented by constraints on the coefficient \(\Param\) by
    \[\norm{\Param_i}_{\ell_1} \leq \ubound^i \quad \text{for each }i=1,2,\ldots, N.\]
    This can be expressed as the intersection of affine constraints which are continuous in \((\Param,t,\Paramw)\).

    \item \label{lemm:point:d} The set \(\adparam\) is compact; see Proposition \ref{prop:comp:conv:adparam:sets}.
    % \todo[fancyline]{Siddhartha asks why?}

    \item \label{lemm:point:e} For fixed \((\ol{\tseq},\ol{\bseq})\), \(\param\in \fsblset\), \(\regupara>0\), and \(i=1,2,\ldots, \dvar\), the sets \(\aset[]{\Param \in \adparam \suchthat \Param \Reg(t^i) \in \admcont}\), \(\aset[]{\Param \in \adparam \suchthat \st(t^i) \in \admst}\), and \(\aset[]{\Param \in \adparam \suchthat \st(\horizon) \in \finset}\) are convex. This is due to the fact that the sets \(\admcont\), \(\admst\), and \(\finset\) are convex and the associated maps \(\Param \mapsto \Param\Reg(t^i)\), \(\Param \mapsto x(t^i;\param,\Param,\Paramw^i)\), and \(\Param \mapsto x(T;\param,\Param,\Paramw^i)\) for each \(i\in \aset[]{1,\ldots,\dvar}\), are linear.
  
    \item \label{lemm:point:f} Assumption \ref{assum:slater-like} ensures the existence of \(\widehat{\Param} \in \adparam\) such that the constraint is strictly feasible at \((\widehat{\Param}, \ol{\tseq}, \ol{\bseq}) \in \adparam \times \totconset^{\dvar}\).
\end{enumerate}
\vspace{2mm}
Then the feasible set mapping of \eqref{eq:g_func} is continuous in the sense of Painlev\'e-Kuratowski \cite[Chapter 3, \S 3.2]{ref:DonRoc-14}. Indeed, this follows immediately from \ref{lemm:point:a}, \ref{lemm:point:c}, and the Slater-like condition in Assumption \ref{assum:slater-like}, using \cite[Example 3B.4]{ref:DonRoc-14}. Moreover, the feasible set mapping of \eqref{eq:g_func} is nonempty and bounded (see Proposition \eqref{prop:unique solution}). This, along with its continuity, implies that the feasible set mapping is Pompeiu-Hausdorff continuous at \((\ol{\tseq},\ol{\bseq})\); see \cite[Theorem 3B.3]{ref:DonRoc-14}. Finally, employing \cite[Theorem \(3\text{B.}5\)]{ref:DonRoc-14} we obtain the continuity of the optimal value mapping \(\gfunc(\xz,\regupara;\cdot,\cdot)\) around \((\ol{\tseq},\ol{\bseq})\). Since \((\ol{\tseq},\ol{\bseq})\) is an arbitrary point, we get continuity of \(\gfunc(\xz,\regupara;\cdot,\cdot)\). This completes the proof of assertion \ref{lemm:exact:sols_1}.

For the second part \ref{lemm:exact:sols_2}, concerning the existence of an optimizer \(\bigl(\tseq^{\ast}(\xz,\ol{\eps}),\bseq^{\ast}(\xz, \ol{\eps})\bigr) \in \totconset^{\dvar}\), the assertion follows readily from the continuity of \(\gfunc(\xz,\regupara;\cdot,\cdot)\) (established in \ref{lemm:exact:sols_1}) and the compactness of \(\totconset^{\dvar}\), appealing to the Weierstrass Theorem \cite[Chapter \(9\), p. \(33\)]{ref:VAZZ-16}.

\end{proof}

% A proof of Proposition \ref{lem:exact:sols} is in Appendix~\ref{appendix:A}.
%%%
\begin{rem}\label{rem:cont:of:F}
\longthmtitle{Significance of~\ref{lemm:exact:sols_1}}
The continuity property \ref{lemm:exact:sols_1} of the optimal value mapping \(\totconset^{\dvar} \ni  (\tseq,\bseq) \mapsto \gfunc(\xz,\ol{\eps};\tseq,\bseq) \in \Rbb\) of the convex program \ref{eq:g_func} is of substantial importance in the selection of algorithms for numerical solutions to \eqref{e:global_max_prob}, and also widens the library of optimization algorithms that may be employed in \eqref{e:global_max_prob}; also see Remark \ref{rem:F:reg:choice} ahead.
\end{rem}

%%%%%%%%%%%%%%%%%
Now we provide the first main result of this article.
\begin{theorem}
\label{thrm:exact:sols}
\longthmtitle{Exact solutions for~\eqref{eq:SR:pre_SIP}}
Consider the OCP \eqref{eq:regularized:SIP} and \ref{eq:g_func} along with their data and suppose that the hypothesis of Proposition \ref{lem:exact:sols} remains valid. Consider the optimization problem \eqref{e:global_max_prob}. Then, with fixed \(\xz \in \fsblset\) and \(\regupara>0\), we have an exact solution, i.e.,
\begin{equation}\label{thrm:exact:sols_3}
\valuefunc(\xz,\ol{\eps})=\gfunc\bigl(\xz,\ol{\eps};\tseq^{\ast}(\xz,\regupara),\bseq^{\ast}(\xz,\regupara)\bigr),
\end{equation}
and the sequence of optimizers \(\bigl(\Param^{\ast}_{\xz, \ol{\eps}}\bigr)_{\ol{\eps}>0}\) defined by
\begin{align*}
\Param^{\ast}_{\xz,\regupara} = & \argmin_{\Param \in \adparam} && \objective(\Param\Reg(\cdot)) + \regupara\regularizer(\Param) \\
        &  \sbjto &&  \begin{cases}\text{the constraints in \eqref{eq:g_func}}\\ 
        \text{hold for }\bigl(\tseq^{\ast}(\xz,\ol{\eps}),\bseq^{\ast}(\xz,\ol{\eps})\bigr),
        \end{cases}
\end{align*}
converges to an optimizer of \eqref{eq:SR:pre_SIP} as \(\ol{\eps} \downarrow 0\).
\end{theorem}

%%%%% Proof of Theorem III.6 %%%%%%
\begin{proof}
To show the equality of optimal values in \eqref{thrm:exact:sols_3}, we will employ \cite[Theorem 1]{ref:DasAraCheCha-22}. To this end, for the CSIP \ref{eq:regularized:SIP}, we verify \cite[Assumptions (1.1)-a--(1.1)-e]{ref:DasAraCheCha-22}:
\begin{enumerate}[label=\textup{T.\alph*)},leftmargin=*]
\item \label{thrm:point:a} The domain \(\adparam\) is closed and convex; see Proposition \ref{prop:comp:conv:adparam:sets}.
\item \label{thrm:point:b} The admissible set of the OCP \eqref{eq:regularized:SIP} is nonempty by Assumption \ref{assum:slater-like}.
\item \label{thrm:point:c} The objective function is convex and continuous in the decision variable \(\Param\); see the arguments established in Proposition \ref{prop:unique solution}.
\item \label{thrm:point:d} \emph{Joint continuity of the constraint maps in the (decision variable, semi-infinite variable)-pair:} The state constraint mappings
\[
 \adparam \times \lcrc{0}{\horizon} \times \adparamD  \ni (\Param,t,\Paramw) \mapsto \begin{pmatrix}
        \st(t; \xz,\Param, \Paramw)\\
        \st(\horizon; \xz,\Param, \Paramw)
   \end{pmatrix}
\]
are jointly continuous; see \ref{lemm:point:c} in the proof of Proposition \ref{lem:exact:sols}.  Recall that the control constraint \(\Param \Reg(t) \in \admcont\) admits the reformulation \(\norm{\Param_i}_{\ell_1} \leq \ubound\) for each \(i=1,2,\ldots, N.\) This is an intersection of affine constraints that are continuous in \((\Param,t,\Paramw)\).
\item \label{thrm:point:e} \emph{Convexity of the constraint maps in the decision variable:}
The maps \begin{align}
\adparam  \ni  \Param \mapsto \begin{pmatrix}
    \st(t; \xz,\Param, \Paramw)\\
    \st(\horizon; \xz,\Param, \Paramw)\\
    \Param \Reg(t)
\end{pmatrix}
\end{align} 
is affine in the decision variable \(\Param\). Consequently, the sets \(\aset[]{\Param \in \adparam \suchthat \Param \Reg(t) \in \admcont}\), \(\aset[]{\Param \in \adparam \suchthat \st(t) \in \admst}\), and \(\aset[]{\Param \in \adparam \suchthat \st(\horizon) \in \finset}\) are convex as they can be expressed as preimages of convex sets under linear/affine mappings. 
\item \label{thrm:point:f} The constraint index set \(\lcrc{0}{\horizon} \times \adparamD\) is compact; see Proposition \ref{prop:comp:conv:adparam:sets}. 
\end{enumerate}
\vspace{2mm}
Assumption \ref{assum:slater-like} along with the preceding points \ref{thrm:point:a}--\ref{thrm:point:f} and \cite[Theorem \(1\)]{ref:DasAraCheCha-22} ensure that with fixed \(\xz \in \fsblset\) and \(\regupara>0\),
\[
\valuefunc(\xz,\ol{\eps})=\gfunc\bigl(\xz,\regupara;\tseq^{\ast}(\xz,\regupara),\bseq^{\ast}(\xz,\regupara)\bigr).
\]
Thus, \eqref{thrm:exact:sols_3} is established. Using the arguments \ref{thrm:point:a}--\ref{thrm:point:f}, the convergence of optimizers as \(\regupara \downarrow 0\), follows immediately by applying \cite[Proposition 3.3]{ref:ParCha-23}. This completes the proof. 
\end{proof}

%%%%% subsec: algo %%%%%%%%
\subsection{\emph{\(\sprob(\param,\regupara)\):} Architecture to solve \eqref{e:global_max_prob}}\label{subsec:sparserob}
We present an algorithmic structure ---  \(\sprob(\cdot,\cdot)\) --- built on Proposition \ref{lem:exact:sols} and Theorem \ref{thrm:exact:sols}, to solve the global optimization problem \eqref{e:global_max_prob}. 
{
\renewcommand{\algorithmcfname}{\(\sprob(\xz,\regupara)\)}
\renewcommand{\thealgocf}{}
\begin{algorithm2e}[!ht]
\SetAlgoLined
\DontPrintSemicolon
\SetKwInOut{ini}{Initialize}
\SetKwInOut{giv}{Data}
\giv{Stopping criterion for recovering the optimizer: \(\textsf{SC}_1(\cdot)\), threshold for \(\textsf{SC}_1\): \(\tau_1\),  Stopping criterion for the global optimization routine: \(\textsf{SC}_2(\cdot)\), threshold for \(\textsf{SC}_2\): \(\tau_2\), Selection criterion for global optimization: $\textsf{SlCr}(\cdot)$, initialization for the regularization parameter \(\eps\), function for implementing regularization $\regul(\cdot,\cdot)$;}
\ini{initialize constraint indices -- \((\tseq_0,\bseq_0) \in \lcrc{\tinit}{\tfin}^{\dvar} \times \adparamD^{\dvar}\),  \(j=0\);}
%
% \BlankLine
%
\While{$\textsf{SC}_1(j)\geqslant \tau_1$}{
\(\regupara \gets \regul(\regupara,j) \)\; 
Set the initial value \(\gfunc^{j}_{\max} = \gfunc(\param,\regupara;\tseq_0,\bseq_0)\)\; 
Set the initial solution \(\ol{\Param}^{j}_{\param,\regupara} = \argmin \gfunc(\param,\regupara;\tseq_0,\bseq_0)\)\;

$k\gets 0$ \;

\While{$\textsf{SC}_2(k)\leqslant \tau_2$}{
\emph{Sample:} \(\tseq^{j,k} \in \lcrc{0}{T}^{\dvar}\);
\(\bseq^{j,k} \in \adparamD^{\dvar}\)\;

\emph{Evaluate:} \(\gfunc^{j,k} \Let \gfunc(\param,\regupara;\tseq^{j,k},\bseq^{j,k})\) as defined in \eqref{eq:g_func} \;      

\emph{Recover the solution} \(\ol{\Param}^{j,k}_{\param,\regupara}\) by solving: 
\[\min_{\Param}\left\{ \hspace{-1mm}\objective\bigl(\Param\Reg(\cdot)\bigr) + \regupara \regularizer(\Param) \hspace{-1mm}\,\middle\vert  \begin{array}{@{}l@{}}
\text{constraints} \\ \text{in } \ref{eq:SR:pre_SIP} \text{ hold } \\ \text{at } (\tseq^{j,k},\bseq^{j,k})        \end{array}
\right\}\]

Update $\gfunc^{j}_{\max}, \ol{\Param}^{j}_{\param,\regupara} \gets $\textsf{SlCr}$(\gfunc^{j}_{\max}, \ol{\Param}^{j}_{\param,\regupara}, \gfunc^{j,k}, \ol{\Param}^{j,k}_{\xz,\regupara},k)$

Update $k \gets k+1$ \;
}\;
Update $j\gets j+1$\;
}
\textbf{Output:} \(\Param\as(\param,\regupara),\Paramw\as(\param,\regupara)\)
\caption{A generic architecture to recover sparse robust optimal control}
\label{alg:sprob}
\end{algorithm2e}
}

\begin{rem}\label{rem:F:reg:choice}
\longthmtitle{Criterion for the choice of global optimization routine}
\emph{\(\sprob(\cdot,\cdot)\)} offers a unified framework for synthesizing constrained sparse and robust control actions: one simply selects an appropriate global optimization routine and uses it to solve the maximization problem \eqref{e:global_max_prob}. An array of such choices can be made depending on the regularity property of the function \(\gfunc(\param,\regupara;\cdot, \cdot)\). 
\begin{itemize}[leftmargin=*,label =\(\circ\)]
    \item If \(\gfunc(\param, \regupara; \cdot, \cdot)\) is continuous, as established in Proposition \ref{lem:exact:sols}, and if the global optimization routine in the algorithm is implemented using the simulated annealing oracle with an appropriate cooling schedule and mild conditions on the transition kernel, then convergence is guaranteed by \cite[Theorem 2]{ref:DasAraCheCha-22} and \cite[Theorem 1]{ref:CJPB-92}.

    \item If \(\gfunc(\param,\regupara;\cdot, \cdot)\) is Lipchitz continuous, then one can employ the global optimization oracles such as \emph{\texttt{SequOOL}} (Sequential Optimistic Optimization with Levels) or \emph{\texttt{LIPO}} (Lipschitz Optimization based on Local Partitions).\footnote{Lipschitz regularity can be established under suitable smoothness assumptions on the constraint sets, as shown in \cite{ref:FiaIsh-90}; we omit the details here.} \emph{\texttt{SequOOL}} uses a deterministic hierarchical scheme \cite{ref:BarGabVal-19} that yields an exponential regret bound \cite{ref:GriValMun-15}. \emph{\texttt{LIPO}} employs stochastic averaging \cite{ref:MalVay-17} and achieves strong PAC-style regret bounds. Both lead to fast convergence to global optima in practice when combined with randomized restarts.
\end{itemize}
 \end{rem}
 %%%%
 \begin{rem}
 \label{rem:on stopping criteria}
 \longthmtitle{Choices of metadata for \(\sprob(\cdot,\cdot)\)}
\emph{\(\sprob(\cdot,\cdot)\)} is written as a general architecture configured for solving OCPs of the form~\eqref{eq:SR-OCP}, in line with the theoretical developments outlined in the preceding sections. Several components of the architecture \emph{\(\sprob(\xz,\regupara)\)} --- such as \emph{\(\textsf{SC}_1(\cdot)\), \(\textsf{SC}_2(\cdot)\), \(\textsf{Regu}(\cdot,\cdot)\)}, and \emph{\(\textsf{SICr}(\cdot)\)} --- are user-defined and customizable, allowing adaptation to the specific global optimization routine being employed. For example:
\begin{itemize}[leftmargin=*]
\item \(\textsf{SC}_1(\cdot)\) denotes the Stopping Criterion that determines when the algorithm should terminate and thereby return an estimate of the optimizer. Since the optimizer is an accumulation point for the optimizer estimates returned from successive outer iterations, \(\textsf{SC}_1(\cdot)\) can, for example, be defined as the norm of the difference between the estimates obtained in successive iterations, i.e., \(\textsf{SC}_1(j)\Let \norm{\ol{\Param}^j_{\param,\regupara}-\ol{\Param}^{j-1}_{\param,\regupara}}\).
\item \(\textsf{Regu}(\cdot,\cdot)\) is the regularization function that updates \(\regupara\), this function must return a decreasing sequence of regularizers for an increase in the outer iteration. For instance, for \(j \in \Nz\) we can pick \(\textsf{Regu}(\regupara,j) = 2^{-j}\regupara\).
\item The inner loop in the algorithm represents a global optimization routine and \(\textsf{SC}_2(\cdot)\) here denotes the termination criterion for this routine. Consider for example that the user picks simulated annealing as the global optimization routine, then \(\textsf{SC}_2\) may be: maximum number of iterations, threshold on minimum temperature value or a bound on the difference between successive optimal values. 
\item Similarly, \(\textsf{SICr}(\cdot)\) denotes the Selection and Improvement Criterion employed by the chosen global optimization routine to determine whether to accept or reject a candidate solution. The specific form of this criterion may depend on factors such as the structure of the problem data or the regularity properties of \(\gfunc(\param, \regupara; \cdot, \cdot)\) as discussed in Remark \ref{rem:F:reg:choice}. For instance, in the case of Simulated Annealing, the criterion takes the following form: the candidate solution is accepted if it improves upon the current best solution, or otherwise with probability \(\exp{\bigl(\frac{\gfunc^{j,k}-\gfunc^{j}_{\max}}{T_m}\bigr)}\), where \(T_m\) is the current temperature.
\end{itemize}
 \end{rem}
%%%%%%%%%%%%%%%%%%%%%%%%%%%%%%%%%%%%%%%
Based on the developments in \S\ref{sec:main_result}, we demonstrate that our results extend to noisy parameter-dependent systems in \S\ref{sec:param:var:systems} and robust minimum attention problems in \S\ref{sec:min:attn}.

\section{Sparse robust control of noisy parameter-dependent system}\label{sec:param:var:systems}
This section treats a class of OCPs where the continuous-time dynamics depend on a parameter in addition to the process noise \(w(\cdot)\). We demonstrate that this problem can also be solved via the techniques employed in \S\ref{sec:main_result}. Let \(\nu \in \N\) and suppose that \(\pset \subset \Rbb^{\nu}\) is a nonempty, compact and convex set of parameters. A commonly occurring example of such uncertainty set \(\pset\) is a polytope, and it is widely found in the robust control literature \cite[Chapter 4]{ref:BoyElGFerBal-94}, \cite[Chapter 8]{ref:DullPag-13}. We consider the systems of the form 
\begin{equation} \label{eq:param:sys}
\dot \st(t) = A(p)\st(t) + B(p)\cont(t) + \dist (t) \,\,\text{for all } t \in \lcrc{\tinit}{\tfin},
\end{equation}
with the same data as given in \ref{ocp:data:1}--\ref{ocp:data:3} and we assume that the mappings \(\pset \ni p \mapsto A(p) \in \Rbb^{\dimst \times \dimst}\) and \( \pset \ni p \mapsto  B(p) \in \Rbb^{\dimst \times \dimcon}\) are continuous. Consequently, adopting the same parametrization of \(\cont(\cdot)\) and \(\dist(\cdot)\) given in \S\ref{sec:prob_form} we directly consider the parametrized and regularized OCP
\begin{equation}
\label{eq:param:OCP} 
\begin{aligned}
& \inf_{\Param \in \adparam} 	&& \objective\bigl(\Param\Reg(\cdot)\bigr) + \regupara \regularizer(\Param)\\
&  \sbjto		&&  \begin{cases}
\text{the constraints in \eqref{eq:SR:pre_SIP}}\\
\text{for all }(t,\Paramw,p) \in \lcrc{\tinit}{\tfin} \times \adparamD \times \pset.
\end{cases}
\end{aligned}
\end{equation}
Note that the solution map \(t\mapsto \st(t) = \st(t;\param,\Param,\Paramw,p)\) depends on the parameter \(p \in \pset\).

Recall the definition of \(\dvar\) in \eqref{eq:dvar}. Let us define \(\pseq \Let (p^1,\ldots,p^{\dvar}) \in \pset^{\dvar}\). To solve the CSIP \eqref{eq:param:OCP} following our preceding developments, we define the relaxed problem
\begin{equation}
    \label{eq:param:g_func}
    \begin{aligned}
        &\hspace{-3mm}\gfunc_p(\xz, \regupara;\tseq,\bseq,\pseq) && \Let\\  &\hspace{-3mm}\inf_{\Param\in \adparam}	&& \hspace{-15mm} \objective\bigl(\Param\Reg(\cdot)\bigr) + \regupara\regularizer(\Param) \\
        &  \hspace{-2mm}\sbjto &&  \hspace{-15mm}\begin{cases}
       \text{the constraints in \eqref{eq:SR:pre_SIP} with }i=1,\ldots,\dvar,\\ \text{for }(\tseq,\bseq,\pseq) \in \lcrc{\tinit}{\tfin}^{\dvar} \times \adparamD^{\dvar} \times \pset^{\dvar}.
        \end{cases}
    \end{aligned}
\end{equation}
The existence and uniqueness of solutions to \eqref{eq:param:OCP} can be established by replicating the steps in the proof of Proposition \ref{prop:unique solution}; we omit the details for brevity.
%%%% rem:on discrete param-var %%%%%%
\begin{rem}
\longthmtitle{On existing literature}
To the best of our knowledge, the only other work on robust optimization-based numerical solutions for sparse robust OCPs is \cite{ref:Scenario:zhang2022sparse}.\footnote{The first author acknowledges Zicheng Zhang for bringing this work to his attention during a visit to Kyoto University.} In \cite{ref:Scenario:zhang2022sparse}, the authors investigated a class of discrete-time, parameter-dependent systems of the form \(\st_{t+1} = A(p)\st(t) + B(p)\cont(t)\) with \(p \in \pset\). They addressed the problem without state and control constraints and developed a scenario programming-based algorithm to design a robust control law that (approximately) minimizes an \(\ell^1\)-objective. Our setting is far more complicated: we consider parameter-dependent and noisy continuous-time systems in the presence of state, action, and disturbance constraints, and in this setting, we offer guarantees of \emph{exact solutions}, whereas the scenario-driven algorithm in \cite{ref:Scenario:zhang2022sparse} can only ensure that the value of a specific relaxed problem (with i.i.d. uncertainty samples) is close to the optimal value with high probability.
\end{rem}
Our second main result given below extends the assertions of Proposition \ref{lem:exact:sols} and Theorem \ref{thrm:exact:sols}. For brevity, we sketch the proof in Appendix~\ref{appendix:A}, referring to the steps already established in the preceding results.
%%% param-var: corollary %%%%%%
\begin{theorem}
\label{thrm:exact:sols:param-var}
\longthmtitle{Exact solutions for \eqref{eq:param:OCP}}
Fix \(\param \in \fsblset\), \(\ol{\eps}>0\) and consider the OCP \eqref{eq:param:OCP} and \eqref{eq:param:g_func} along with their data. Suppose that a suitably modified version of the feasibility hypothesis in Assumption \ref{assum:slater-like} holds. Define \(\totconset^{\dvar} \Let \lcrc{0}{\horizon}^{\dvar}\times \adparamD^{\dvar} \times \pseq^{\dvar}\). Consider the optimization problem
 \begin{equation}\label{e:global_max_prob_param-var}
\sup_{(\tseq,\bseq,\pseq)} \aset[\big]{\gfunc_p(\param,\regupara;\tseq,\bseq,\pseq) \suchthat (\tseq,\bseq,\pseq) \in \totconset^{\dvar}}.
\end{equation}
Then \(\totconset^{\dvar} \ni (\tseq,\bseq,\pseq) \mapsto \gfunc_p\bigl(\xz,\regupara;\tseq,\bseq,\pseq\bigr)\) is continuous and there exists \(\bigl(\tseq^{\ast}(\xz,\regupara),\bseq^{\ast}(\xz,\regupara),\pseq^{\ast}(\xz,\regupara)\bigr)\) that solves \eqref{e:global_max_prob_param-var}. Moreover, we have exact solutions: \(\valuefunc(\xz,\ol{\eps})=\gfunc\bigl(\xz,\regupara;\tseq^{\ast}(\xz,\regupara),\bseq^{\ast}(\xz,\regupara),\pseq^{\ast}(\xz,\regupara)\bigr)\), and the sequence of optimizers \(\bigl(\Param^{\ast}_{\param,\regupara}\bigr)_{\regupara>0}\) defined by
\begin{align*}
\Param^{\ast}_{\param,\regupara} = & \argmin_{\Param \in \adparam} &&  \objective(\Param\Reg(\cdot)) + \regupara \regularizer(\Param) \\
        &  \hspace{2mm}\sbjto &&  \begin{cases}\text{the constraints in \eqref{eq:param:OCP}}\\ 
        \text{hold for }\bigl(\tseq^{\ast}(\param,\regupara),\bseq^{\ast}(\param,\regupara),\pseq^{\ast}(\xz,\regupara)\bigr),
        \end{cases}
\end{align*}
converges as \(\regupara \downarrow 0\) to an optimizer of \eqref{eq:param:OCP} with \(\regupara =0\) in the objective.
\end{theorem}
Note that the architecture \(\sprob(\cdot,\cdot)\) introduced in \S\ref{subsec:sparserob} carries over with only minor modifications.

\subsection{Discussion and placement of our results}\label{sec:discussion}
Sparse robust signal processing techniques have advanced substantially over the past decade, and a plethora of techniques from robust optimization are employed in such problems \cite{ref:EldYonBenNem-04,ref:XiaGuiXin-14}. Consider the following prototypical problem in this field: For \(q,\nu,\nu' \in \N\), consider the optimization problem
\begin{align}
\label{eq:minmax:signal:pros}
    &\min_{v \in \Rbb^{\nu} } && \hspace{-1cm} \norm{v}_1  \nn \\
    & \sbjto && \hspace{-1cm} H(\xi)v \le \ell(\xi) \text{ for all }\xi \in \Xi,
\end{align}
where \(\Xi \subset \Rbb^{q}\) is a nonempty set of parameters/uncertainties, \(H: \Xi \lra \Rbb^{\nu' \times \nu}\) and \(\ell:\Xi \lra \Rbb^{\nu'}\) are continuous functions. The problem \eqref{eq:minmax:signal:pros} bears strong similarity with the problems of the type \eqref{eq:SR:pre_SIP} and \eqref{eq:param:OCP} treated in this article, and it is natural to ponder whether the tools from robust optimization alluded to the above are applicable for solving these problems. The answer is `no', for reasons we explain in detail below. 
\begin{enumerate}[label=\Roman*.]
\item Notice that in \eqref{eq:SR:pre_SIP}/\eqref{eq:param:OCP}, parametric uncertainties enter the dynamics through the matrices \(A\) and \(B\) as detailed in \S\ref{sec:param:var:systems}. They enter the inequality constraints in \eqref{eq:SR:pre_SIP}/\eqref{eq:param:OCP} in the form 
\begin{align}
    H(p)\Param \le \ell(p) \quad\text{for all }p \in \pset \nn
\end{align}
and for all other semi-infinite parameters, where \(H(p)\) contains terms of the form \(\epower{A(p)t}\), \(\int_{0}^{\horizon}\epower{A(p)(t-s)}(B(p)\Param \Reg(s) + \Paramw\RegD(s))\odif{s}\), etc. The dependence of \(H(p)\) on \(p\) is, consequently, rather complicated, and certainly not of the ``affine'' type \(H(p) = A'p + b'\) for some matrices \(A'\) and \(b'\). The tools from robust optimization employed in robust signal processing rely on parametric inequalities that are typically ``affine'' in the parameters; as such they do not apply to the context of control problems of the type \eqref{eq:SR:pre_SIP}/\eqref{eq:param:OCP}. In fact, this is a key point of departure of robust control problems of the type \eqref{eq:SR:pre_SIP}/\eqref{eq:param:OCP} from the types of constraints in robust signal processing problems. 
\item Moreover, even arriving at conservative approximations to the constraints in \eqref{eq:SR:pre_SIP} and \eqref{eq:param:OCP} by inequalities of the type \(H(p)\Param \le \ell(p)\) that are affine in \(p\) is nontrivially challenging due to the appearance of \(p\) in the matrix exponential as indicated above, and may lead quickly to infeasibility of the resulting problem. Our results in \S\ref{sec:main_result} and \S\ref{sec:param:var:systems} are, in contrast, exact and devoid of any conservatism. 
\end{enumerate}

\section{Robust minimum attention problem}\label{sec:min:attn}
Minimum attention problems \cite{ref:brockett1997minimum} are optimal control problems where the rate of change in control actions remains bounded. Consider the dynamical system 
\begin{align}\label{eq:max:attn:dyn}
    \dot{\st}(t)  =  A \st(t) + b u(t) + w(t) \quad\text{for a.e. }t \in \lcrc{0}{\horizon}
\end{align}
where \(x(t) \in \Rbb^d\), \(u(t) \in \Rbb\), and \(w(t) \in \Rbb^d\) for all \(t\in \lcrc{0}{\horizon}\); \(A\in \Rbb^{d \times d}\) and \(b\in \Rbb^d\). For simplicity, we do not consider parameter-dependent systems here. Consider the robust version of the minimum attention problem in \cite[\S 3]{MN-DN:20}
\begin{align}
\label{eq:minimum:attention}
&\min_{\cont(\cdot)} && \int_{0}^{\horizon}\norm{\dot{\cont}(t)}_{\ell_1}\odif{t} \\
&  \sbjto		&&  \begin{cases}
\text{dynamics}\,\,\eqref{eq:max:attn:dyn},\,\st(0)=\xz,\,\st(\tfin) \in \finset,\\
\cont(0) = \nu,\, \cont(\tfin) = 0, \abs{\dot{\cont}(t)}\leq r, \nn\\
\text{for all }(t,w(t)) \in \lcrc{\tinit}{\tfin} \times \Wbb.
\end{cases}
\end{align}
We will assume that the OCP \ref{eq:minimum:attention} is well-posed. Introducing a new variable \(\lcrc{0}{\horizon} \ni t \mapsto v(t) \Let \dot{u}(t) \in \Rbb\), we obtain an augmented system
\begin{equation}
\label{eq:minimum:aug:system}
\dot{z}(t)=\bar{A} z(t)+\bar{B} v(t) + \bar{c} \dist(t) \quad\text{for all }t \in \lcrc{0}{\horizon},
\end{equation}
where the matrices in \eqref{eq:minimum:aug:system} are given by
\[
z(t) \Let \begin{pmatrix} \st(t) \\ \cont(t)
\end{pmatrix},
\; \bar{A} \Let \begin{pmatrix}
A & b \\
0 & 0
\end{pmatrix}, \; \bar{B} \Let \begin{pmatrix} 0 \\ 1
\end{pmatrix}, \; \bar{c} \Let \begin{pmatrix}I_d \\ 0 \end{pmatrix},
\]
with \(I_d\) is the standard \(d \times d\) identity matrix. In \eqref{eq:minimum:aug:system}, the control \(\lcrc{0}{\horizon} \ni t \mapsto v(t)\) is computed based on the initial state observation \(x(0) \Let \param\) and the initial control value \(\cont(0) \Let \nu\). Adopting the same parametrization for the control \(v(\cdot)\) and disturbance \(\dist(\cdot)\) as in \S\ref{sec:prob_form}, while retaining the same notations for the associated coefficients and the terminal set, we arrive at the parametrized and regularized CSIP formulation of the OCP \eqref{eq:minimum:attention}
\begin{align}\label{eq:min:attn:SIP}
&\min_{\Param} && \objective(\Param\Reg(\cdot)) + \regupara \regularizer(\Param)\\
&\sbjto && \begin{cases}
\text{dynamics }\eqref{eq:minimum:aug:system},\; z_{1}(0)= \param,\\
z_{2}(0) = \nu, \; z_{1}(T)\in \finset, \; z_{2}(T) = 0 \nn\\
\abs{\Param\Reg(t)} \le r \; \text{for all } (t,\Paramw) \in \lcrc{\tinit}{\tfin} \times \adparamD.
\end{cases}
\end{align}
With a slight abuse of notation, we shall continue to denote the value function associated with \ref{eq:min:attn:SIP} by \(\valuefunc(\cdot,\cdot)\), and for any \(\regupara>0\), refer to the relaxed problem as \(\gfunc(\param,\regupara; \tseq, \bseq)\), following the convention established in \eqref{eq:g_func}, with the necessary adaptations understood. This leads immediately to the following corollary, whose proof proceeds analogously to that of Theorem \ref{thrm:exact:sols}. Note that here \(\dvar \Let N \).
%%%% corollary %%%%
\begin{corr}
\label{corr:exact:sols}
\longthmtitle{Exact solutions for \eqref{eq:min:attn:SIP}}
Fix \(\param \in \fsblset\), \(\ol{\eps}>0\) and consider the OCP \eqref{eq:minimum:attention} and \eqref{eq:min:attn:SIP} along with its data. Suppose that the feasibility hypothesis in Assumption \ref{assum:slater-like} holds with obvious modifications. Define \(\totconset^{\dvar} \Let \lcrc{0}{\horizon}^{\dvar}\times \adparamD^{\dvar}\). Consider the optimization problem
\begin{equation}\label{e:global_max_prob_MinAttn}
\sup_{(\tseq,\bseq)} \aset[\big]{\gfunc(\param,\regupara;\tseq,\bseq) \suchthat (\tseq,\bseq) \in \totconset^{\dvar}}.
\end{equation}
Then \(\totconset^{\dvar} \ni (\tseq,\bseq) \mapsto \gfunc\bigl(\xz,\regupara;\tseq,\bseq\bigr)\) is continuous and there exists \(\bigl(\tseq^{\ast}(\xz,\regupara),\bseq^{\ast}(\xz,\regupara)\bigr)\) that solves \eqref{e:global_max_prob_MinAttn}. Moreover, \(\valuefunc(\xz,\ol{\eps})=\gfunc\bigl(\xz,\regupara;\tseq^{\ast}(\xz,\regupara),\bseq^{\ast}(\xz,\regupara)\bigr)\), and the sequence of optimizers \(\bigl(\Param^{\ast}_{\param,\regupara}\bigr)_{\regupara>0}\) defined by
\begin{align*}
\Param^{\ast}_{\param,\regupara} = & \argmin_{\Param \in \adparam} && \objective(\Param\Reg(\cdot)) + \regupara \regularizer(\Param) \\
        &  \hspace{2mm}\sbjto &&  \begin{cases}\text{the constraints in \eqref{eq:min:attn:SIP}}\\ 
        \text{hold for }\bigl(\tseq^{\ast}(\param,\regupara),\bseq^{\ast}(\param,\regupara)\bigr),
        \end{cases}
\end{align*}
converges as \(\regupara \downarrow 0\) to an optimizer of \eqref{eq:min:attn:SIP} with \(\regupara =0\) in the objective.
\end{corr}
%%%%%%%%
\begin{proof}
The proof follows in a manner analogous to that of Theorem \ref{thrm:exact:sols} and Theorem \ref{thrm:exact:sols:param-var} and is therefore omitted.
\end{proof}

\section{Numerical experiments}\label{sec:NumExp}

In this section, we present numerical experiments to demonstrate the applicability of our approach for two continuous-time control problems: the sparse robust control problem (with and without parametric uncertainties) and the noisy attention problem. For the sparse robust control problem, we present our experimental results obtained using the \(\sprob(\cdot,\cdot)\) architecture for the benchmark spring-mass-damper system. The global optimization routine used in \(\sprob(\cdot,\cdot)\) for our experiments was the simulated annealing algorithm \cite[Chapter 13]{ref:OH-02}, \cite{ref:CJPB-92}. The metadata for this was chosen in accordance with Remark~\ref{rem:on stopping criteria}. The noisy attention problem is considered for the spring-mass-damper system with only process noise. All numerical experiments were conducted in Julia v1.1.14 and executed on a laptop equipped with an AMD Ryzen 5 5600H processor and 8 GB of RAM.

\subsection{Sparse robust control}
\subsubsection*{Spring-mass-damper with external disturbance}
\label{sec:smd:dis}
Consider the continuous-time model of the spring-mass-damper system with disturbance~\cite[Chapter 3]{KO:01}
\begin{align}
\label{eq:num:smd}
\dot{\st}(t)  =  \begin{pmatrix}0& 1 \\ 
-2& -0.5
\end{pmatrix}  \st(t) + \begin{pmatrix}
    0 \\ 1
\end{pmatrix} u(t) + w(t)
\end{align}
for all \(t \in \lcrc{0}{\horizon}\). We have the following problem data 
\begin{align}
    \label{eq:num:prob:data}
   \hspace{-2mm} \begin{cases}
       \horizon=5,\,\st(0) \Let (-1, 1)^{\top},\,u(t) \in \lcrc{-1}{1},\\ 
       \st(t) \in \admst \Let  \aset[]{(\dummyx_1,\dummyx_2)^{\top}\in \Rbb^2 \suchthat \abs{\dummyx_1}\leq 0.7},\\ 
       \st(\horizon)\in
       \finset \Let \left\{(\dummyx_1,\dummyx_2)^{\top}\in \Rbb^2 \;\middle\vert\;\begin{array}{@{}l@{}}\abs{\dummyx_1}\leq 0.3 \text{ and}\\ \abs{\dummyx_2}\leq 0.3\end{array}
\right\},\\
        \text{and } w(t) \in \lcrc{-0.4}{0.4}^2\text{ for all } t \in \lcrc{0}{\horizon}.
    \end{cases}
\end{align}

The control and process noise trajectories \(\cont(\cdot)\) and \(w(\cdot)\) for the system \eqref{eq:num:smd} were parametrized using piecewise constant functions, in accordance with Definitions~\ref{defn:discrete_admcon} and~\ref{defn:discrete_admdist} with \(N= M=250\). Recall the expression in \eqref{eq:simplified_cost}; and consider the OCP
\begin{align}
\label{eq:ocp:smd:dis}
&\min_{\Param} && \objective(\Param\Reg(\cdot)) \nn\\
&\sbjto &&\begin{cases}
\text{dynamics }\eqref{eq:num:smd} \text{ and the data \eqref{eq:num:prob:data}},\\
\text{for all }(t,\beta)\in \lcrc{\tinit}{\tfin} \times \adparamD.
\end{cases}
\end{align}
We solved the OCP~\eqref{eq:ocp:smd:dis} employing the \(\sprob(\cdot,\cdot)\) architecture. In Fig.~\ref{fig:msap}, we depict the control trajectories generated for various values of \(\regupara\). Observe that as the value of  \(\regupara\) decreases to zero, the generated control trajectories become closer and closer to each other. The sparse robust control trajectory in Fig. \ref{fig:msap} corresponds to \(\regupara=10^{-6}\) and it can be observed from Fig. \ref{fig:msap} that the control input is sparse and becomes active primarily when the states approach either the path constraint boundaries or the terminal constraint set. The state trajectories corresponding to 10000 distinct realizations of the disturbance, generated under the sparse control strategy, are shown in Fig.~\ref{fig:smd:dis:traj}. In the plot 9990 out of 10000, i.e., \(99.9\%\) of the state trajectories satisfy the state constraints and reach the terminal constraint set robustly. The observed violations stem from unavoidable numerical errors in the computer simulations. We also present several simulation results for a comparative illustration employing the scenario optimization \cite{MC-SG:18}. The method proceeds via generating a finite number of samples of the semi-infinite constraints and then solving the ensuing finitely-constrained OCP via off-the-shelf solvers whose solution is then used as the robust sparse solution to generate the state trajectories. For the illustrations we are considering the spring mass damper with external disturbances (31) with the problem data (32). We used 1000 and 5000 scenarios for the experiments. It can be clearly observed in the plots Fig.~\ref{fig:sc1000} and Fig.~\ref{fig:sc5000} that multiple trajectories have violated the state constraints and terminal constraints.
\begin{figure*}
\begin{subfigure}[t]{0.33\textwidth}
\includegraphics[width=\linewidth]{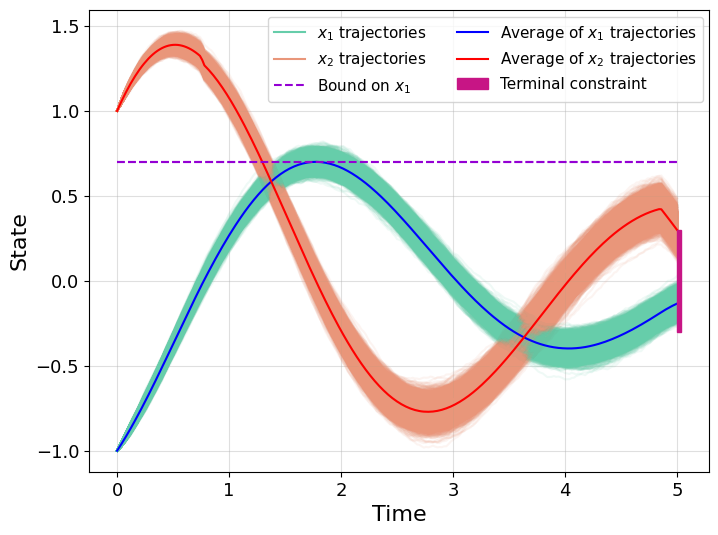}   
\caption{State trajectories with non-robust control.}
\label{fig:smd:no:dis}
\end{subfigure}
\begin{subfigure}[t]{0.33\textwidth}
\includegraphics[width=\linewidth]{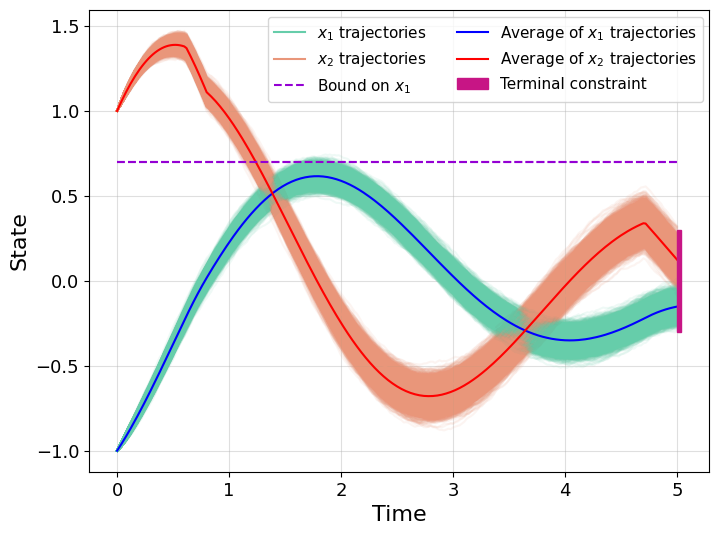}
\caption{State trajectories with the control generated using 1000 scenarios.}
\label{fig:sc1000}
\end{subfigure}
\begin{subfigure}[t]{0.33\textwidth}
\includegraphics[width=\linewidth]{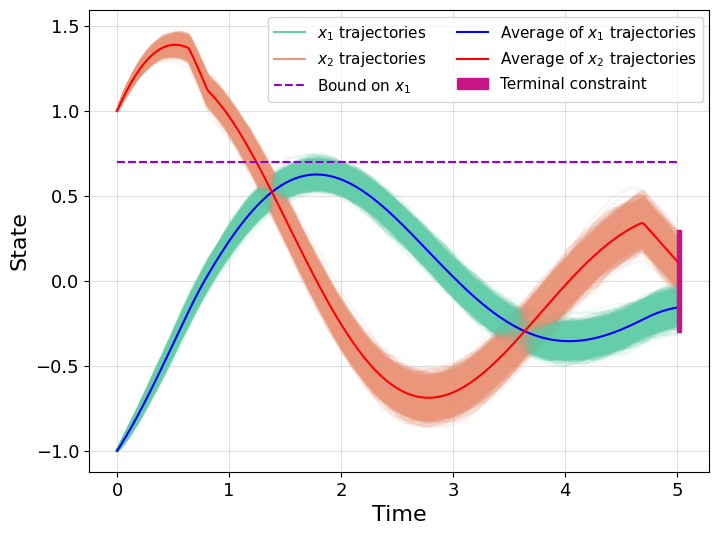}
\caption{State trajectories with the control generated using 5000 scenarios.}
\label{fig:sc5000}
\end{subfigure}
\caption{The left-hand subfigure (Fig.~\ref{fig:smd:no:dis}) illustrates the evolution of 10,000 state trajectories of the spring–mass–damper system~\eqref{eq:num:smd} under non-robust sparse optimal control. The results indicate that neglecting external disturbances leads to violations of the state and terminal constraints. Subfigures Fig.~\ref{fig:sc1000} and Fig.~\ref{fig:sc5000} show the state trajectories obtained using control computed via the scenario approach with 1,000 and 5,000 scenarios, respectively. The scenario-based control significantly reduces constraint violations, but state and terminal constraint violations remain due to the probabilistic nature of the approach.
}
\label{fig:smd:msap}
\end{figure*}

\begin{figure*}
\begin{subfigure}{0.49\textwidth}    
\includegraphics[width=\linewidth]{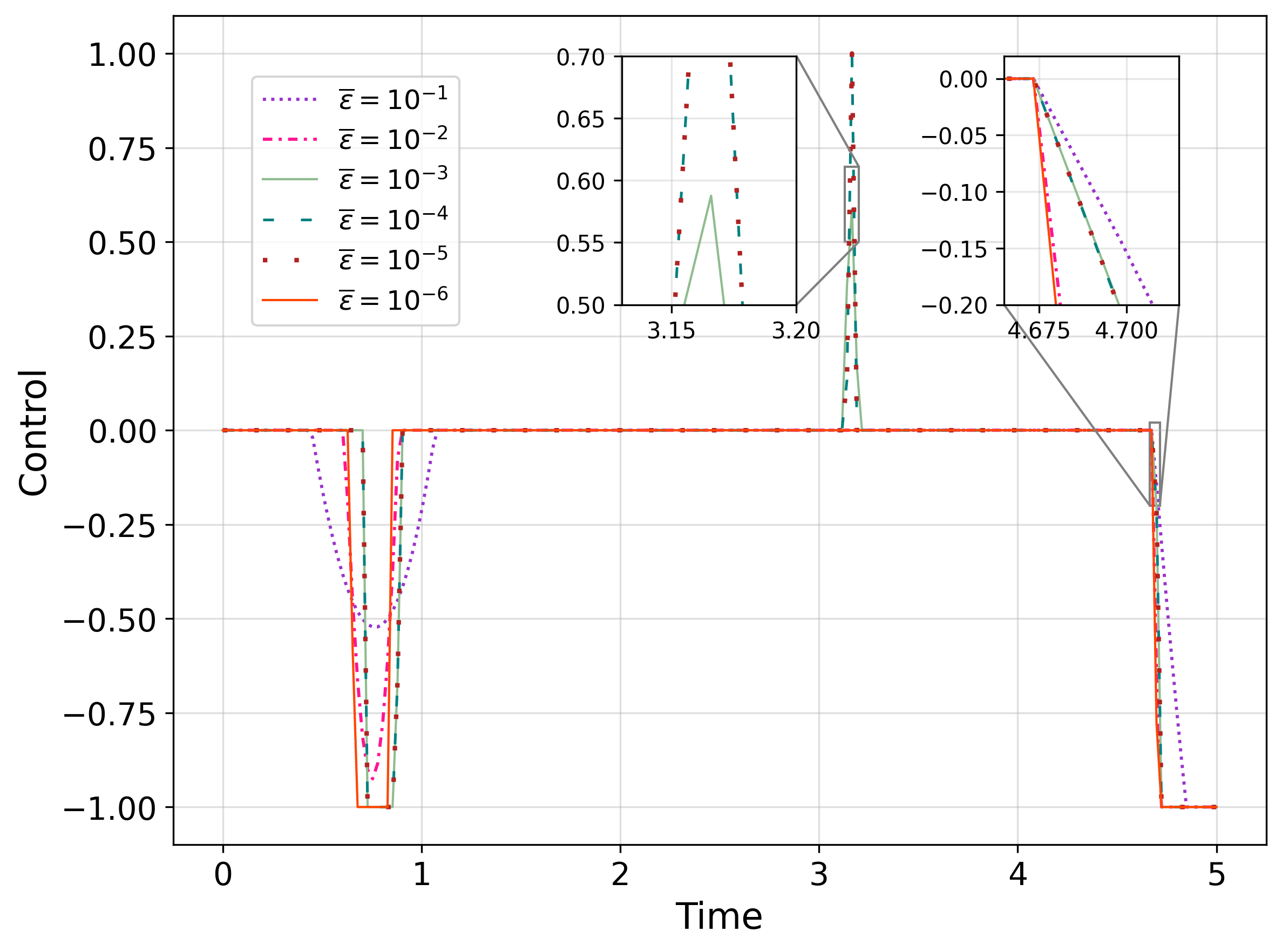}
\caption{Control trajectories with different values of the parameter \(\regupara\).}
\label{fig:msap}
\end{subfigure}
\begin{subfigure}{0.49\textwidth}
% \includegraphics[width=\linewidth]{plots/smd_dis_inp_new.png}       
% \caption{Control trajectory.}
% \label{fig:smd:dis:control}
\includegraphics[width=\linewidth]{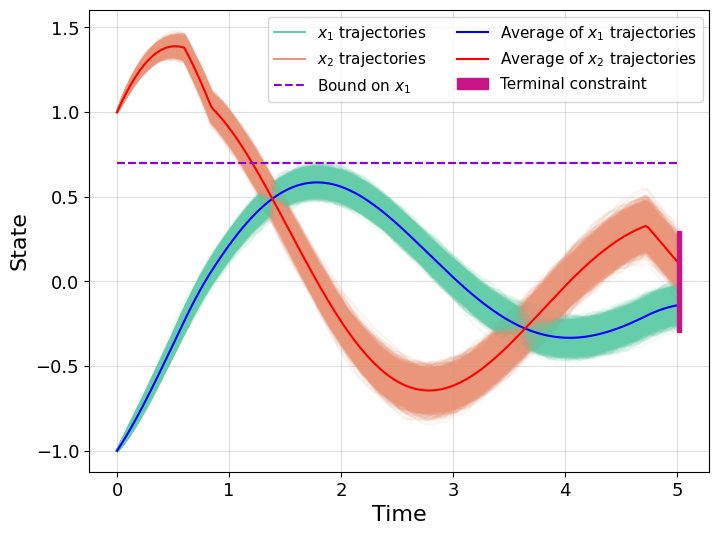}
\caption{State trajectories.}
% \label{fig:smd:dis:control}
\label{fig:smd:dis:traj}
\end{subfigure}
\caption{In Fig.~\ref{fig:msap} (the left-hand subfigure), control trajectories obtained using the \(\sprob(\cdot,\cdot)\) architecture, with different values of the parameter \(\regupara\) are shown. These trajectories correspond to the spring-mass-damper system~\eqref{eq:num:smd}, and are computed by solving the sparse robust problem~\eqref{eq:ocp:smd:dis} for various values of \(\regupara\). As \(\regupara\) decreases, these trajectories exhibit increasing similarity, and a clear trend toward sparsity emerges. In Fig.~\ref{fig:smd:dis:traj}, state trajectories of the spring-mass-damper system~\eqref{eq:num:smd}, generated using 10,000 different realizations of the process noise are shown. The robust control computed by solving~\eqref{eq:ocp:smd:dis} via \(\sprob\bigl((-1, 1)^{\top},10^{-6}\bigr)\) is used. Also, 9990 out of 10000 (i.e., \(99.9\%\)) state trajectories, generated via this sparse control, robustly satisfy both the state and terminal constraints in the presence of process noise. The violations are attributed to inevitable numerical inaccuracies arising from computer simulations.}
\end{figure*}

\subsubsection*{Spring-mass-damper with parametric uncertainties}
\label{sec:smd:par}
Consider the following representation of the spring-mass-damper system containing parametric uncertainties 
\begin{align}
\label{eq:num:smd:par}
\dot{\st}(t)  =  \begin{pmatrix}0& 1 \\ 
-2+p_1& -0.5+p_2
\end{pmatrix}  \st(t) + \begin{pmatrix}
    0 \\ 1+p_3
\end{pmatrix} u(t)
\end{align}
for all \(t \in \lcrc{0}{\horizon}\). We have the following problem data 
\begin{align}
    \label{eq:num:prob:par:data}
    \hspace{-2mm}\begin{cases}
       \horizon=5,\,\st(0) \Let (-1, -1)^{\top},\,u(t) \in \lcrc{-1}{1}, \\ 
       \st(t) \in \admst \Let \aset[]{(\dummyx_1,\dummyx_2)^{\top}\in \Rbb^2 \suchthat \abs{\dummyx_1}\leq 0.5} \\ 
        \st(\horizon)\in
       \finset \Let \left\{(\dummyx_1,\dummyx_2)^{\top}\in \Rbb^2 \;\middle\vert\;\begin{array}{@{}l@{}}\abs{\dummyx_1}\leq 0.1 \text{ and}\\ \abs{\dummyx_2}\leq 0.1\end{array}
\right\},\\
        \text{and } p \in \lcrc{-0.1}{0.1}^3\text{ for all } t \in \lcrc{0}{\horizon},
    \end{cases}
\end{align}
where \(p \Let (p_1,p_2,p_3)^{\top}\) is the tuple that represents the parametric uncertainty. Similar to the previous case, the control trajectory \(\cont(\cdot)\) for the system \eqref{eq:num:smd:par} was parametrized using piecewise constant functions, in accordance with Definition~\ref{defn:discrete_admcon} with \(N=250\). With this parameterization, we considered the OCP
\begin{align}
\label{eq:ocp:smd:par}
&\min_{\Param} && \objective(\Param\Reg(\cdot)) \nn\\
&\sbjto &&\begin{cases}
\text{dynamics }\eqref{eq:num:smd:par} \text{ and the data \eqref{eq:num:prob:par:data}},\\
\text{for all }(t,\Paramw,p) \in \lcrc{\tinit}{\tfin} \times \adparamD \times \pset,
\end{cases}
\end{align}
and employed the \(\sprob(\cdot,\cdot)\) architecture. The results of this numerical experiment are presented in Fig.~\ref{fig:smd:par}. Figure~\ref{fig:smd:par:control} depicts the sparse robust control input, while Fig.~\ref{fig:smd:par:traj} illustrates the state trajectories corresponding to 10000 different realizations of the parametric uncertainty. These plots highlight the sparsity of the control and confirm that the state trajectories satisfy both the state and terminal constraints despite the presence of parametric uncertainties.

\begin{figure*}
\begin{subfigure}{0.49\textwidth}
\includegraphics[width=\linewidth]{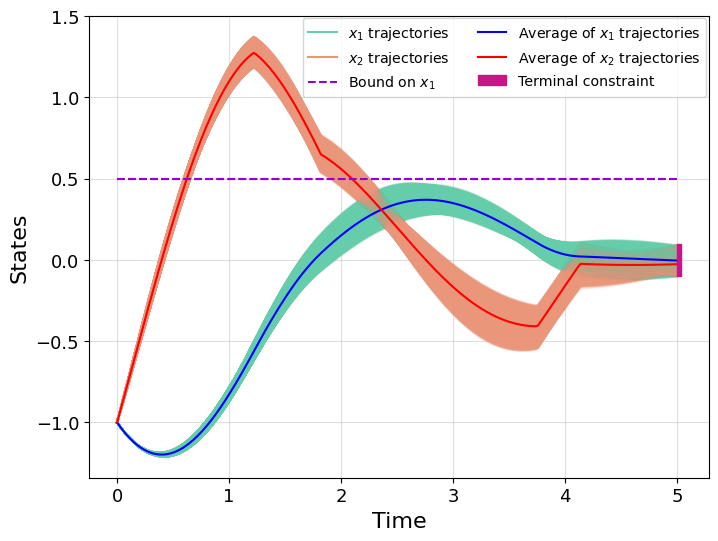}   
\caption{State trajectories.}
\label{fig:smd:par:traj}
\end{subfigure}
\begin{subfigure}{0.49\textwidth}
\includegraphics[width=\linewidth]{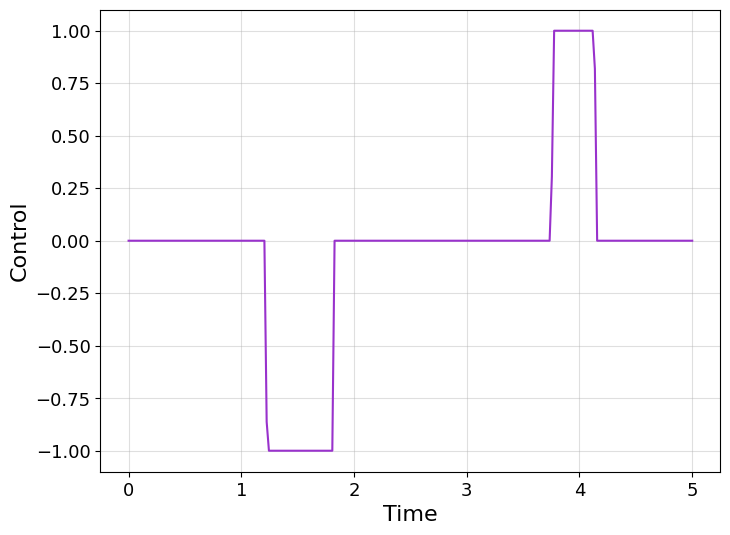}    
\caption{Control trajectory.}
\label{fig:smd:par:control}       
\end{subfigure}
\caption{The state trajectories of the spring-mass-damper system~\eqref{eq:num:smd:par}, subject to 10000 different realizations of the uncertain parameter, are shown in Fig.~\ref{fig:smd:par:traj}. The system is controlled using the sparse robust control computed by solving~\eqref{eq:ocp:smd:par} via \(\sprob\bigl((-1, -1)^{\top},10^{-6}\bigr)\) architecture. The figure illustrates the effectiveness of the proposed approach in generating sparse robust control inputs for the system with data~\eqref{eq:num:prob:par:data}. As observed from Fig.~\ref{fig:smd:par:control}, the control input is sparse, i.e., being active only over a small portion of the time horizon. Despite the presence of parametric uncertainties, the resulting state trajectories satisfied both the state and terminal constraints, demonstrating the robustness of the proposed control scheme.}
\label{fig:smd:par}
\end{figure*}

\subsection{Robust minimum attention problem}\label{sec:NumExp:sparse:minattn}
For the minimum attention problem described in \S\ref{sec:min:attn}, we consider the following augmented dynamics of the spring-mass-damper system in~\eqref{eq:num:smd}, in accordance with the formulation in~\eqref{eq:minimum:aug:system}
\begin{equation}
\label{eq:num:min}
\dot{z}(t)=\begin{pmatrix}
 0 & 1 & 0 \\ -2& -0.5& 1 \\ 0& 0& 0
\end{pmatrix} z(t)+\begin{pmatrix}
0 \\ 0\\ 1
\end{pmatrix} v(t) + \begin{pmatrix}
1 & 0 \\ 0 & 1\\ 0 & 0
\end{pmatrix} \dist(t)
\end{equation}
for all \(t \in \lcrc{0}{\horizon}\), where \(t \mapsto z(t) \Let (\st_1(t), \st_1(t), u(t))^{\top}\). For this problem we define \(\finset \Let \aset[]{(\dummyx_1,\dummyx_2)^{\top}\in \Rbb^2 \suchthat \abs{\dummyx_1}\leq 0.25 \text{ and }\abs{\dummyx_2}\leq 0.25}\), keeping all other problem data from~\eqref{eq:num:prob:data} intact. Moreover, a bound was imposed on the input rate of change, \(\abs{\dot{\cont}(t)}\leq 2\). The trajectories for the rate of change in input \(v(\cdot)\) and the disturbance \(w(\cdot)\) were parameterized according to the Definition~\ref{defn:discrete_admcon} and Definition~\ref{defn:discrete_admdist}. This leads to the following OCP, which was solved using the \(\sprob(\cdot,\cdot)\) architecture
\begin{align}
\label{eq:min:attn:SIP_ne}
&\min_{\Param} && \objective(\Param\Reg(\cdot)) \\
&\sbjto && \begin{cases}
\text{dynamics }\eqref{eq:num:min},\; z(0)= (-1,1,0)^{\top},\\
(z_{1}(T),z_{2}(T))^{\top}\in \finset, \; z_{3}(T) = 0, \\
\st(\horizon) \in \finset,\,|\Reg(t) \Param| \le 2,\\ \text{for all } (t,\Paramw) \in \lcrc{\tinit}{\tfin} \times \adparamD.
\end{cases} \nn
\end{align}
Fig.~\ref{fig:atten:rate} shows the trajectory of the rate of change of the control input, corresponding to \(\regupara=10^{-6}\). Given that the goal of the minimum attention problem is to induce sparsity in the control rate signal, Fig. \ref{fig:atten:rate} demonstrates the emergence of such sparse behavior. The control trajectory, shown in Fig. \ref{fig:atten:input}, was used to simulate the response of the spring-mass-damper system \eqref{eq:num:smd} under 10000 different realizations of the disturbance signal. The resulting state trajectories are presented in Fig.~\ref{fig:atten:traj}.
\begin{figure*}
\begin{subfigure}{0.33\textwidth}
\includegraphics[width=\linewidth]{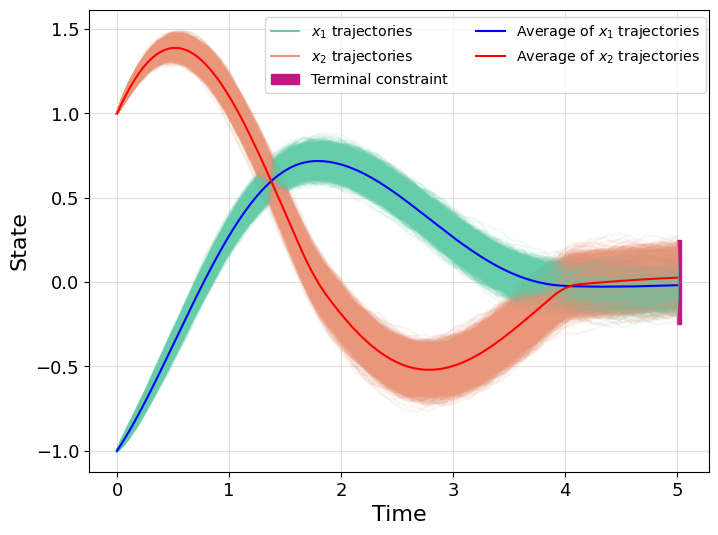}
\caption{State trajectories.}
\label{fig:atten:traj}
\end{subfigure}
\begin{subfigure}{0.33\textwidth}
\includegraphics[width=\linewidth]{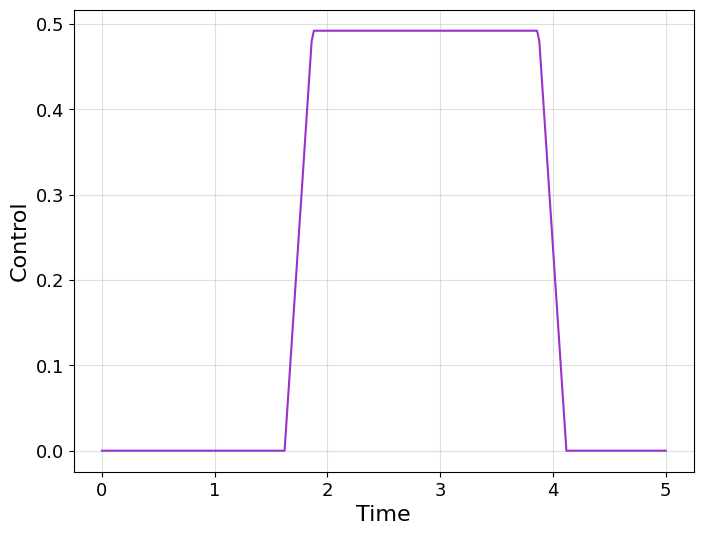} 
\caption{Control trajectory.}
\label{fig:atten:input}
\end{subfigure}
\begin{subfigure}{0.33\textwidth}
\includegraphics[width=\linewidth]{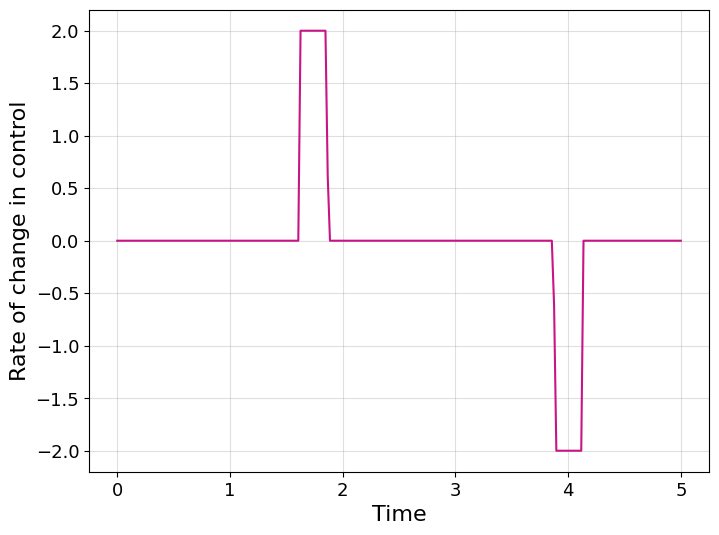}
\caption{Rate of change in control.}
\label{fig:atten:rate}
\end{subfigure}
\caption{Fig.~\ref{fig:atten:traj} shows the evolution of 10000 trajectories generated for the spring-mass-damper system~\eqref{eq:num:smd} with data~\eqref{eq:num:prob:data} using the control obtained by solving the minimum attention problem~\eqref{eq:min:attn:SIP_ne} via the~\(\sprob\bigl((-1,1)^{\top},10^{-6}\bigr)\) approach. This control is depicted in Fig.~\ref{fig:atten:input} and the corresponding rate of change in control is is shown in Fig.~\ref{fig:atten:rate}. The attention problem demands sparsity in the rate of change in control, this can be can be observed in Fig.~\ref{fig:atten:rate}.}
\label{fig:atten}
\end{figure*}

\section{Conclusion}
\label{sec:conclusion}
This article introduced an algorithmic framework for constrained sparse robust optimal control, which offered strong theoretical guarantees of exact solutions and ensured satisfaction of an uncountable family of constraints in a computationally viable fashion. We demonstrated our results for a class of noisy parameter-dependent systems and noisy minimum attention control problems. Numerical illustrations on a benchmark control problem demonstrated the effectiveness of our technique. Possible future directions are: (a) developing faster gradient-based algorithms; (b) investigating other sparsity promoting structures \cite{ref:Ike-24}, for example quasi-convex norms.

% references %
\bibliographystyle{ieeetr}
\bibliography{refs}

@book{ref:DonRoc-14,
    AUTHOR = {A. L. Dontchev and R. T. Rockafellar},
     TITLE = {Implicit {F}unctions and {S}olution {M}appings},
    SERIES = {Springer Series in Operations Research and Financial Engineering},
   EDITION = {2nd},
 PUBLISHER = {Springer, New York},
      YEAR = {2014},
     PAGES = {xxviii+466},
      NOTE = {doi: \url{https://doi.org/10.1007/978-1-4939-1037-3}},
}

@misc{ref:SidKen-26,
      title={Robust maximum hands-off optimal control: existence, maximum principle, and ${L}^{0}$-${L}^1$ equivalence}, 
      author={S. Ganguly and K. Kashima},
      year={2026},
      eprint={2601.07256},
      archivePrefix={arXiv},
      primaryClass={math.OC},
      NOTE = {doi: \url{
https://doi.org/10.48550/arXiv.2601.07256}}, 
}

@article{ref:ExaEvaPan-18,
  title={Stochastic \({L}^1\)-optimal control via forward and backward sampling},
  author={I. Exarchos and E. A. Theodorou and P. Tsiotras},
  journal={Systems \& Control Letters},
  volume={118},
  pages={101--108},
  year={2018},
  publisher={Elsevier},
  NOTE = {doi: \url{https://doi.org/10.1016/j.sysconle.2018.06.005}}
}

@book {ref:MV:20,
    AUTHOR = {M. Vidyasagar},
     TITLE = {An {I}ntroduction to {C}ompressed {S}ensing},
    SERIES = {Computational Science \& Engineering},
    VOLUME = {22},
 PUBLISHER = {Society for Industrial and Applied Mathematics (SIAM),
              Philadelphia, PA},
      YEAR = {[2020] \copyright 2020},
     PAGES = {xi+341},
      ISBN = {978-1-611976-11-3},
   MRCLASS = {94-02 (05C90 15A83 65F99 94A12)},
  MRNUMBER = {4050519},
}

@book{ref:FolReal-99,
  title={Real {A}nalysis: {M}odern {T}echniques and their {A}pplications},
  author={G. B. Folland},
  year={1999},
  publisher={John Wiley \& Sons}
}

@article{ref:DasAraCheCha-22,
	author = {S. Das and A. Aravind and A. Cherukuri and D. Chatterjee},
	title = {Near-optimal solutions of convex semi-infinite programs via targeted sampling},
	journal = {Annals of Operations Research},
	note = {doi: \url{https://doi.org/10.1007/s10479-022-04810-4}},
	year = {2022},
}

@article{ref:ParCha-23,
title = {Attaining the Chebyshev bound for optimal learning: A numerical algorithm},
journal = {Systems \& Control Letters},
volume = {181},
number = {105648},
year = {2023},
issn = {0167-6911},
note = {doi: \url{https://doi.org/10.1016/j.sysconle.2023.105648}},
author = {P. Paruchuri and D. Chatterjee}
}

@INPROCEEDINGS{ref:SG:NR:DC:RB-22,
  author={S. Ganguly and N. Randad and D. Chatterjee and R. Banavar},
  booktitle={2022 IEEE 61st Conference on Decision and Control (CDC)}, 
  title={Constrained trajectory synthesis via quasi-interpolation}, 
  year={2022},
  volume={},
  number={},
  pages={4533-4538},
NOTE = {doi: \url{https://doi.org/10.1109/CDC51059.2022.9992892}}
}

@ARTICLE{ref:gansilcha24,
   author={S. Ganguly and R. A. D'Silva and D. Chatterjee},
  journal={IEEE Transactions on Automatic Control}, 
  title={{Q}u{ITO} v.2: Numerical Solutions with Uniform Error Guarantees to Optimal Control Problems under Path Constraints}, 
  year={2026},
  volume={},
  number={},
  pages={1-16},
 }

@book{ref:betts-book,
    author = {J. T. Betts},
     title = {Practical {M}ethods for {O}ptimal {C}ontrol and {E}stimation {U}sing {N}onlinear {P}rogramming},
    series = {Advances in Design and Control},
    volume = {1},
 publisher = {SIAM},
     pages = {xiv+434},
     year = {2010},
NOTE = {doi: \url{https://doi.org/10.1137/1.9780898718577}}
}

@article{ref:MN:MaxHandsOff,
  title={Maximum hands-off control: a paradigm of control effort minimization},
  author={M. Nagahara and D. E. Quevedo and D. Ne{\v{s}}i{\'c}},
  journal={IEEE Transactions on Automatic Control},
  volume={61},
  number={3},
  pages={735--747},
  year={2015},
  publisher={IEEE},
NOTE = {doi: \url{https://doi.org/10.1109/TAC.2015.2452831}}
}

@article{ref:DC:MaxHandsOff:SCL,
  title={Characterization of maximum hands-off control},
  author={D. Chatterjee and M. Nagahara and D. E. Quevedo and KS. M. Rao},
  journal={Systems \& Control Letters},
  volume={94},
  pages={31--36},
  year={2016},
  publisher={Elsevier},
NOTE = {doi: \url{https://doi.org/10.1016/j.sysconle.2016.05.002}}
}

@article{ref:SparseApp:Space,
  title={Terminal spacecraft rendezvous and capture with {LASSO} model predictive control},
  author={E. N. Hartley and M. Gallieri and J. M. Maciejowski},
  journal={International Journal of Control},
  volume={86},
  number={11},
  pages={2104--2113},
  year={2013},
  publisher={Taylor \& Francis},
NOTE = {doi: \url{https://doi.org/10.1080/00207179.2013.789608}}
}

@article{ref:SparseApp:NetworkPackets,
  title={Sparse packetized predictive control for networked control over erasure channels},
  author={M. Nagahara and D. E. Quevedo and J. {\O}stergaard},
  journal={IEEE Transactions on Automatic Control},
  volume={59},
  number={7},
  pages={1899--1905},
  year={2013},
  publisher={IEEE},
NOTE = {doi: \url{https://doi.org/10.1109/TAC.2013.2294622}}
}

@inproceedings{ref:SparseApp:SensingFeedDisgn,
  title={Sparsity based feedback design: A new paradigm in opportunistic sensing},
  author={S. Bhattacharya and T. Ba{\c{s}}ar},
  booktitle={Proceedings of the 2011 American Control Conference},
  pages={3704--3709},
  year={2011},
  organization={IEEE},
NOTE = {doi: \url{https://doi.org/10.1109/ACC.2011.5991014}}
}

@article{ref:SparseApp:PDEs:I,
  title={Approximation of sparse controls in semilinear equations by piecewise linear functions},
  author={E. Casas and R. Herzog and G. Wachsmuth},
  journal={Numerische Mathematik},
  volume={122},
  number={4},
  pages={645--669},
  year={2012},
  publisher={Springer},
NOTE = {doi: \url{https://doi.org/10.1007/s00211-012-0475-7}}
}

@article{ref:SparseApp:PDE:II,
  title={Second-order and stability analysis for state-constrained elliptic optimal control problems with sparse controls},
  author={E. Casas and F.Tr\'{o}ltzsch},
  journal={SIAM Journal on Control and Optimization},
  volume={52},
  number={2},
  pages={1010--1033},
  year={2014},
  publisher={SIAM},
NOTE = {doi: \url{https://doi.org/10.1137/130917314}}
}

@article{ref:SparseApp:Donoho,
  title={Compressed sensing},
  author={D. L. Donoho},
  journal={IEEE Transactions on Information Theory},
  volume={52},
  number={4},
  pages={1289--1306},
  year={2006},
  publisher={IEEE},
NOTE = {doi: \url{https://doi.org/10.1109/TIT.2006.871582}}
}

@book{ref:SparseBook:Vidyasagar,
  title={An {I}ntroduction to {C}ompressed {S}ensing},
  author={M. Vidyasagar},
  year={2019},
  publisher={SIAM},
NOTE = {doi: \url{https://doi.org/10.1137/1.9781611976120}}
}

@book{ref:SparseBook:Unser,
  title={An {I}ntroduction to {S}parse {S}tochastic {P}rocesses},
  author={M. Unser and P. D. Tafti},
  year={2014},
  publisher={Cambridge University Press},
NOTE = {doi: \url{https://doi.org/10.1017/CBO9781107415805}}
}

@article{SparseBook:ML:StatLearning,
  title={Statistical {L}earning with {S}parsity: {T}he {LASSO} and {G}eneralizations},
  author={T. Hastie and R. Tibshirani and M. Wainwright},
  journal={Monographs on Statistics and Applied Probability},
  volume={143},
  pages={8},
  year={2015},
NOTE = {doi: \url{https://doi.org/10.1201/b18401}}
}

@inproceedings{ref:Scenario:zhang2022sparse,
  title={Sparse robust control design via scenario program},
  author={Z. Zhang and Y. Fujisaki},
  booktitle={Proceedings of the ISCIE International Symposium on Stochastic Systems Theory and its Applications},
  pages={61--64},
  year={2022},
NOTE = {doi: \url{https://doi.org/10.5687/sss.2022.61}}
}

@INPROCEEDINGS{MN-DN:20,
  author={M. Nagahara and D. Nešić},
  booktitle={2020 59th IEEE Conference on Decision and Control (CDC)}, 
  title={An Approach to Minimum Attention Control by Sparse Derivative}, 
  year={2020},
  volume={},
  number={},
  pages={5005-5010},
  NOTE = {doi: \url{https://doi.org/10.1109/CDC42340.2020.9303783}}
}

@inproceedings{ref:brockett1997minimum,
  title={Minimum attention control},
  author={Brockett, W},
  booktitle={Proceedings of the 36th IEEE Conference on Decision and Control},
  volume={3},
  pages={2628--2632},
  year={1997},
  organization={IEEE},
NOTE = {doi: \url{https://doi.org/10.1109/CDC.1997.657776}}
}

@book {ref:VAZZ-16,
    AUTHOR = {V. A. Zorich},
     TITLE = {Mathematical {A}nalysis. {II}},
    SERIES = {Universitext},
   EDITION = {2nd Edition},
 PUBLISHER = {Springer-Verlag, Berlin},
      YEAR = {2016},
     PAGES = {xx+720},
      ISBN = {978-3-662-56966-5},
      NOTE = {doi: \url{https://link.springer.com/book/10.1007/978-3-662-48993-2}},
}

@book{ref:MH-book-20,
  title={Sparsity {M}ethods for {S}ystems and {C}ontrol},
  author={M. Nagahara},
  year={2020},
  publisher={Now Publishers},
NOTE = {doi: \url{http://dx.doi.org/10.1561/9781680837254}}
}

@book {ref:Rud-Analysis,
    AUTHOR = {Rudin, Walter},
     TITLE = {{P}rinciples of {M}athematical {A}nalysis},
    SERIES = {International Series in Pure and Applied Mathematics},
   EDITION = {Third},
PUBLISHER = {McGraw-Hill Book Co., New York-Auckland-D\"usseldorf},
      YEAR = {1976},
     PAGES = {x+342},
   MRCLASS = {26-02},
  MRNUMBER = {385023},
}

@book{KO:01,
author = {K. Ogata},
title = {Modern Control Engineering},
year = {2001},
publisher = {Prentice Hall PTR},
address = {USA},
edition = {4th},
NOTE = {URL: \url{https://dl.acm.org/doi/10.5555/516039}}
}

@article{ref:CJPB-92,
  title={Convergence theorems for a class of simulated annealing algorithms on \(\mathbb{R}^d\)},
  author={C. J. P. B{\'e}lisle},
  journal={Journal of Applied Probability},
  volume={29},
  number={4},
  pages={885--895},
  year={1992},
  publisher={Cambridge University Press},
NOTE = {doi: \url{https://doi.org/10.2307/3214721}}
}

@book{ref:OH-02,
  title={Finite {M}arkov {C}hains and {A}lgorithmic {A}pplications},
  author={O. Haggstrom},
  series={London Mathematical Society Student Texts},
  year={2002},
  publisher={Cambridge University Press},
NOTE = {doi: \url{https://doi.org/10.1017/CBO9780511613586}}
}

@InProceedings{ref:BarGabVal-19,
  title = 	 {A simple parameter-free and adaptive approach to optimization under a minimal local smoothness assumption},
  author =       {P. L. Bartlett and V. Gabillon and M. Valko},
  booktitle = 	 {Proceedings of the 30th International Conference on Algorithmic Learning Theory},
  pages = 	 {184--206},
  year = 	 {2019},
  editor = 	 {Garivier, Aurélien and Kale, Satyen},
  volume = 	 {98},
  series = 	 {Proceedings of Machine Learning Research},
  month = 	 {22--24 Mar},
  publisher =    {PMLR},
 
  NOTE = 	 {doi: \url{https://proceedings.mlr.press/v98/bartlett19a.html}}
}

@article{ref:FiaIsh-90,
    AUTHOR = {A. V. Fiacco and {Yo}. Ishizuka},
     TITLE = {Sensitivity and stability analysis for nonlinear programming},
   JOURNAL = {Annals of Operations Research},
    VOLUME = {27},
      YEAR = {1990},
    NUMBER = {1-4},
     PAGES = {215--235},
      NOTE = {doi: \url{https://doi.org/10.1007/BF02055196}},
}

@inproceedings{ref:GriValMun-15,
 author = {J-B. Grill and M. Valko and R. Munos},
 booktitle = {Advances in Neural Information Processing Systems},
 editor = {C. Cortes and N. Lawrence and D. Lee and M. Sugiyama and R. Garnett},
 pages = {},
 publisher = {Curran Associates, Inc.},
 title = {Black-box optimization of noisy functions with unknown smoothness},
NOTE ={URL: \url{https://proceedings.neurips.cc/paper_files/paper/2015/file/ab817c9349cf9c4f6877e1894a1faa00-Paper.pdf}},
 volume = {28},
 year = {2015},
}

@InProceedings{ref:MalVay-17,
title = 	 {Global optimization of {L}ipschitz functions},
author =       {C. Malherbe and N. Vayatis},
pages = 	 {2314--2323},
year = 	 {2017},
editor = 	 {Precup, Doina and Teh, Yee Whye},
volume = 	 {70},
booktitle = 	 {Proceedings of Machine Learning Research},
publisher =    {PMLR},
NOTE = {URL: \url{https://proceedings.mlr.press/v70/malherbe17a.html}},
}

@book{ref:BenElGNem-09,
    AUTHOR = {A. Ben-Tal and L. El Ghaoui and A. Nemirovski},
     TITLE = {Robust {O}ptimization},
    SERIES = {Princeton Series in Applied Mathematics},
PUBLISHER = {Princeton University Press, Princeton, NJ},
      YEAR = {2009},
     PAGES = {xxii+542},
      NOTE = {doi: \url{https://doi.org/10.1515/9781400831050}},
}

@book{ref:DullPag-13,
  title={A {C}ourse in {R}obust {C}ontrol {T}heory: {A} {C}onvex {A}pproach},
  author={G. E. Dullerud and F. Paganini},
  volume={36},
  year={2013},
  publisher={Springer Science \& Business Media}, 
Series = {Texts in Applied Mathematics},
NOTE = {doi: \url{https://doi.org/10.1007/978-1-4757-3290-0}}
}

@book{ref:BoyElGFerBal-94,
    AUTHOR = {S. Boyd and L. El Ghaoui and E. Feron and
              V. Balakrishnan},
     TITLE = {Linear {M}atrix {I}nequalities in {S}ystem and {C}ontrol {T}heory},
    SERIES = {SIAM Studies in Applied Mathematics},
    VOLUME = {15},
PUBLISHER = {Society for Industrial and Applied Mathematics (SIAM),
              Philadelphia, PA},
      YEAR = {1994},
     PAGES = {xii+193},
      NOTE = {doi: \url{https://doi.org/10.1137/1.9781611970777}},
}

@book{ref:AliTsa-21,
    AUTHOR = {A. R. Alimov and I. G. Tsar'kov},
     TITLE = {Geometric {A}pproximation {T}heory},
    SERIES = {Springer Monographs in Mathematics},
 PUBLISHER = {Springer, Cham},
      YEAR = {2021},
     PAGES = {xxi+508},
      NOTE = {doi: \url{https://doi.org/10.1007/978-3-030-90951-2}},
}

@article{ref:EldYonBenNem-04,
  title={Robust mean-squared error estimation in the presence of model uncertainties},
  author={Y. C. Eldar and A. Ben-Tal and A. Nemirovski},
  journal={IEEE Transactions on Signal Processing},
  volume={53},
  number={1},
  pages={168--181},
  year={2004},
  publisher={IEEE},
NOTE = {doi: \url{https://doi.org/10.1109/TSP.2004.838933}}
}

@article{ref:XiaGuiXin-14,
  title={Robust compressed sensing with bounded and structured uncertainties},
  author={X. Qing and G. Hu and X. Wang},
  journal={IET Signal Processing},
  volume={8},
  number={7},
  pages={783--791},
  year={2014},
  publisher={Wiley Online Library},
NOTE = {doi: \url{https://doi.org/10.1049/iet-spr.2013.0260}}
}

@article{ref:PolSch-05,
  title={Hard problems in linear control theory: Possible approaches to solution},
  author={B. T. Polyak and P.S. Shcherbakov},
  journal={Automation and Remote Control},
  volume={66},
  number={5},
  pages={681--718},
  year={2005},
  publisher={Springer},
NOTE = {doi: \url{https://doi.org/10.1007/s10513-005-0115-0}}
}

@article{ref:PolakMayneSIPsInControl,
  title={Control system design via semi-infinite optimization: a review},
  author={E. Polak and D. Q. Mayne and D. M. Stimler},
  journal={Proceedings of the IEEE},
  volume={72},
  number={12},
  pages={1777--1794},
  year={1984},
  publisher={IEEE},
NOTE = {doi: \url{https://doi.org/10.1109/PROC.1984.13086}}
}

@ARTICLE{ref:Ike-24,
  author={T. Ikeda},
  journal={IEEE Transactions on Automatic Control}, 
  title={Nonconvex Optimization Problems for Maximum Hands-Off Control}, 
  year={2025},
  volume={70},
  number={3},
  pages={1905-1912},
  NOTE = {doi: \url{10.1109/TAC.2024.3474061}}
}

@article{ref:ItoIkeKashi-21,
title = {Sparse optimal stochastic control},
journal = {Automatica},
volume = {125},
pages = {109438},
year = {2021},
issn = {0005-1098},
author = {K. Ito and T. Ikeda and K. Kashima},
NOTE = {doi: \url{https://doi.org/10.1016/j.automatica.2020.109438}},
}

@article{ref:minmax:ness:vinter,
  title={Minimax optimal control},
  journal={SIAM journal on control and optimization},
author = {R. B. Vinter},
  volume={44},
  number={3},
  pages={939--968},
  year={2005},
  publisher={SIAM},
NOTE = {doi: \url{https://doi.org/10.1137/S0363012902415244}},
}

@article{ref:elango2024successive,
title = {Continuous-time successive convexification for constrained trajectory optimization},
journal = {Automatica},
volume = {180},
pages = {112464},
year = {2025},
author={P. Elango and D. Luo and A. G. Kamath and S. Uzun and T. Kim and B. B. A{\c{c}}{\i}kme{\c{s}}e}
}

@inproceedings{ref:SD:SG:AM:DC:CtmpcCDC,
  title={Towards continuous-time {MPC}: A novel trajectory optimization algorithm},
  author={S. Das and S. Ganguly and A. Muthyala and D. Chatterjee},
  booktitle={2023 62nd IEEE Conference on Decision and Control (CDC)},
  pages={3276--3281},
  year={2023},
NOTE = {doi: \url{10.1109/CDC49753.2023.10383236}},
}

@inproceedings{ref:fazlyab2016interior,
  title={Interior point method for dynamic constrained optimization in continuous time},
  author={M. Fazlyab and S. Paternain and V. M. Preciado and A. Ribeiro},
  booktitle={2016 American Control Conference (ACC)},
  pages={5612--5618},
  year={2016},
  organization={IEEE},
NOTE = {doi: \url{https://doi.org/10.1109/ACC.2016.7526550}},
}

@book{MC-SG:18,
    AUTHOR = {M. C. Campi and S. Garatti},
    TITLE = {Introduction to the Scenario Approach},
    SERIES = {MOS-SIAM Series on Optimization},
    VOLUME = {26},
    PUBLISHER = {Society for Industrial and Applied Mathematics (SIAM), Philadelphia, PA;
                 Mathematical Optimization Society, Philadelphia, PA},
    YEAR = {2018},
    PAGES = {viii+116},
    ISBN = {978-1-61197-544-4},
    MRCLASS = {90C15 (62H30 90-02 93-02 93E35)},
    MRNUMBER = {3909429},
    MRREVIEWER = {Kurt Marti},
    NOTE = {doi: \url{10.1137/1.9781611975444}},
}

% \appendices

\section*{Appendix A: Proof of Theorem \ref{thrm:exact:sols:param-var}}
\label{appendix:A}
\begin{proof}
The continuity of \(\gfunc_p\bigl(\xz,\regupara;\cdot,\cdot,\cdot\bigr)\) for fixed \(\param \in \fsblset\) and \(\regupara>0\) follows from similar arguments furnished in Proposition \ref{lem:exact:sols}. Indeed, \ref{lemm:point:a}, \ref{lemm:point:b}, \ref{lemm:point:d}, \ref{lemm:point:f} in the proof of Proposition \ref{lem:exact:sols} holds with obvious modifications. Recall that the solution of \eqref{eq:param:sys}, after parameterizing the control and the disturbance appropriately, admits the compact form
    \begin{align*}
        & t \mapsto \st(t;\xz,\Param,\Paramw,p)\\
        & = \epower{A(p)t}\overline{x} + \sum_{i=1}^{K}\biggl(\int_{s_{i-1}}^{s_i}\epower{A(p)(t-s)}\odif{s}\biggr)(B(p)\Param\Reg_i+\Paramw\RegD_i)
    \end{align*}
and thus, the joint continuity of the map
    \begin{align*}
        \adparam \times \lcrc{0}{\horizon} \times \adparamD \times \pset \ni (\Param,t,\Paramw,p) \mapsto \begin{pmatrix}
        \st(t; \xz,\Param, \Paramw,p)\\
        \st(\horizon; \xz,\Param, \Paramw,p)
   \end{pmatrix},
    \end{align*}
    is guaranteed in view of the continuity of 
     \begin{align*}
         (\Param,t,\Paramw,p) \mapsto \begin{pmatrix}
             A(p),
             B(p),
             \epower{A(p)t}
         \end{pmatrix}. 
     \end{align*}
Thus, we have the analogous argument of \ref{lemm:point:c} in the proof of Proposition \ref{lem:exact:sols}. Concerning \ref{lemm:point:e}, observe that for fixed \((\ol{\tseq},\ol{\bseq},\ol{\pseq})\), \(\param\in \fsblset\), \(\regupara>0\), and \(i=1,2,\ldots, \dvar\), the sets \(\aset[]{\Param \in \adparam \suchthat \Param \Reg(t^i) \in \admcont}\), \(\aset[]{\Param \in \adparam \suchthat \st(t^i;\param,\Param^i,\Paramw^i,p^i) \in \admst}\), and \(\aset[]{\Param \in \adparam \suchthat \st(\horizon;\param,\Param^i,\Paramw^i,p^i) \in \finset}\) are convex because the sets \(\admcont\), \(\admst\), and \(\finset\) are convex and the maps \(\Param \mapsto \Param\Reg(t^i)\), \(\Param \mapsto x(t^i;\param,\Param^i,\Paramw^i,p^i)\), and \(\Param \mapsto x(T;\param,\Param^i,\Paramw^i,p^i)\) are, for each \(i\in \aset[]{1,\ldots,\dvar}\), linear and thus continuous and convex. The continuity of \(\gfunc_p\bigl(\xz,\regupara;\cdot,\cdot,\cdot\bigr)\) is continuous, and the existence follows immediately from the Weierstrass theorem. 

We draw attention to the fact that pointers \ref{thrm:point:a}--\ref{thrm:point:f} in the proof of Theorem \ref{thrm:exact:sols} are directly applicable, and the joint continuity of the constraints mappings were shown in the preceding paragraph. Appealing to \cite[Theorem \(1\)]{ref:DasAraCheCha-22} we have  
\begin{equation*}         \valuefunc(\xz,\ol{\eps})=\gfunc_p\bigl(\xz,\regupara;\tseq^{\ast}(\xz,\regupara),\bseq^{\ast}(\xz,\regupara),\pseq^{\ast}(\xz,\regupara)\bigr),
\end{equation*} 
and applying \cite[Proposition 3.3]{ref:ParCha-23} we get convergence of optimizers as \(\regupara \downarrow 0\). The proof is complete. 
\end{proof}

\end{document}